\documentclass{article}[12pt]
\usepackage[utf8]{inputenc}
\usepackage{braket}
\usepackage{slashed}
\usepackage{graphicx}
\usepackage{amsmath}
\usepackage{amssymb}
\usepackage{amsthm}
\usepackage{marginnote}
\usepackage{xcolor}
\usepackage{hyperref}
\usepackage{multicol}
\usepackage{float}
\usepackage{hyperref}
\usepackage{cancel}
\hypersetup{
    colorlinks=true,
    linkcolor=blue,
    filecolor=magenta,      
    urlcolor=cyan,
    citecolor=red,
    pdftitle={berndt},
    pdfpagemode=FullScreen,
    }
\usepackage[margin=0.75in]{geometry}

\newtheorem{lemma}{Lemma}
\newtheorem{remark}{Remark}
\newtheorem{proposition}{Proposition}
\newtheorem{theorem}{Theorem}
\newtheorem{corollary}{Corollary}

\title{Berndt-type Integrals: Unveiling Connections with Barnes Zeta and Jacobi Elliptic Functions}
\author{Zachary P. Bradshaw$^1$ and Christophe Vignat$^{2,3}$}
\date{
    $^1$Naval Surface Warfare Center, Panama City Division, USA\\
    $^2$LSS, CentraleSup\'elec, Universit\'e Paris Saclay, France\\
    $^3$Tulane University, Department of Mathematics, New Orleans, USA
}

\begin{document}

\maketitle

\begin{abstract}
    We address a class of definite integrals known as Berndt-type integrals, highlighting their role as specialized instances within the integral representation framework of the Barnes-zeta function. Building upon the foundational insights of Xu and Zhao, who adeptly evaluate these integrals using rational linear combinations of Lambert-type series and derive closed-form expressions involving products of $\Gamma^4(1/4)$ and $\pi^{-1}$, we uncover direct evaluations of the Barnes-zeta function. Moreover, our inquiry leads us to establish connections between Berndt-type integrals and Jacobi elliptic functions, as well as moment polynomials investigated by Lomont and Brillhart, a relationship elucidated through the seminal contributions of Kuznetsov. In this manner, we extend and integrate these diverse mathematical threads, unveiling deeper insights into the intrinsic connections and broader implications of Berndt-type integrals in special function and integration theory.
\end{abstract}

\section{Introduction} 

Since its inception, the theory of integration has captivated practitioners of mathematics. Unlike the derivative operation, which relies solely on local information and adheres to a finite set of rules applicable to any differentiable function, definite integration hinges on non-local data spanning an entire interval. As a result, integration theory encompasses a diverse array of techniques, each tailored to tackle specific challenges, and yet no single combination suffices to evaluate every conceivable integral. This inherent complexity renders the theory fertile ground for exploration, replete with unsolved problems which the practitioner may spend a lifetime resolving.

Here we concern ourselves with the class
\begin{equation}\label{eq:berndt-int}
    I_{\pm}(s,p)=\int_0^\infty\frac{x^{s-1}}{(\cosh(x)\pm\cos(x))^p}dx
\end{equation}
of \textit{Berndt-type integrals} of order $p$ \cite{berndt2016}, a class with no table historical roots dating back to the pioneering work of the esteemed mathematician S. Ramanujan in 1916 \cite{ramanujan1916}. These integrals have also featured prominently in the context of moment problems studied by Ismail and Valent \cite{ismail1998}, and the ongoing effort to evaluate them has been advanced by Xu and Zhao \cite{xu2023,xu2024}. Central to the methodology for evaluating these integrals is the application of Cauchy’s residue theory, a formidable tool in one's integration toolkit. In this way, Berndt-type integrals of varying orders reveal their equivalence to series involving hyperbolic trigonometric functions, and by leveraging results attributed to Ramanujan, explicit closed-form expressions for these sums emerge.

In this work, we give an alternative evaluation of the Berndt-type integrals in terms of the Barnes zeta function \cite{barnes1901,barnes1904,ruijsenaars2000,friedman2004} defined by
\begin{equation*}
    \zeta_N(s,w\lvert a_1,\ldots,a_N)=\sum_{n_1\ge 0,\ldots,n_N\ge0}\frac{1}{(w+n_1a_1+\cdots+n_Na_N)^s}
\end{equation*}
and of its alternating version
\begin{equation*}
    \tilde{\zeta}_N(s,w\lvert a_1,\ldots,a_N)=\sum_{n_1\ge 0,\ldots,n_N\ge0}\frac{(-1)^{n_1+\dots +n_N}}{(w+n_1a_1+\cdots+n_Na_N)^s}.
\end{equation*}
While the Lambert series representation given by Xu and Zhao allows for a closed form evaluation in many cases, the Barnes zeta representation presented here gives the evaluation of the integral for \textit{any} choice of $s,p$ in the domain of convergence but does not produce a closed form in terms of more elementary functions. Combining this with the results of Xu and Zhao produces closed form evaluations of the Barnes zeta function. Our results are easily generalized to the case of Dirichlet type analogs of the Barnes zeta function. We also relate Berndt-type integrals to Jacobi elliptic functions through the generating function methods of Kuznetsov \cite{kuznetsov} and further connect them to a class of moment polynomials \cite{Lomont}. The work of Lomont and Brillhart then produces recurrence relations for Berndt-type integrals.

In Section~\ref{sec:barnes}, we evaluate Berndt-type integrals using the Barnes zeta function and observe that this procedure generalizes to Dirichlet type analogs of this function. In the appendix, we take an alternative approach involving Euler-Barnes and Bernoulli-Barnes polynomials \cite{bayad2013,dixit2012,jiu2016} and identify the Bernoulli-Barnes polynomials as an analytic continuation of Berndt-type integrals to negative powers. In Section~\ref{sec:jacobi} we review the work of Kuznetsov connecting Berndt-type integrals to Jacobi elliptic functions before using it to draw connections to a class of moment polynomials defined by a recurrence relation. Kuznetsov's direct evaluation of $I_+(s,1)$ is then extended to the case of $I_-(s,1)$ in Section~\ref{sec:jacobi-extension}. Finally, we exhibit more numerical properties of Kuznetsov's integrals in Section \ref{section:Arithmetic} before giving concluding remarks in Section~\ref{sec:conclusion}.

\section{Connections with the Barnes Zeta Function}\label{sec:barnes}
\subsection{An integral representation}
We begin with a classic integral representation of the Barnes zeta function, which appears as \cite[(3.2)]{ruijsenaars2000}. Its proof, here reproduced for completeness, follows as a consequence of Euler's integral.


\begin{proposition}\label{prop:general}
Let $\Re(s)>N$, $\Re(w)>0$, and $\Re(a_j)>0$ for $j=1,\ldots,N$. Then
\begin{align*}
    \zeta_N(s,w\lvert a_1,\ldots,a_N)=\frac{1}{\Gamma(s)}\int_0^\infty u^{s-1}e^{-wu}\prod_{j=1}^N(1-e^{-a_ju})^{-1}\ du.
\end{align*} 
The alternating version is
\begin{align*}
    \widetilde{\zeta}_N(s,w\lvert a_1,\ldots,a_N)=\frac{1}{\Gamma(s)}\int_0^\infty u^{s-1}e^{-wu}\prod_{j=1}^N(1+e^{-a_ju})^{-1}\ du. \end{align*} 
\end{proposition}
\begin{proof} Since $\Re(s)>N$, we can write the gamma function in its integral representation, and the series representation of $\zeta_N$ converges absolutely for $\Re(w)>0$ and $\Re(a_j)>0$. We have
\begin{align}
    \zeta_N(s,w\lvert a_1,\ldots,a_N)\Gamma(s)&=\sum_{n_1,\ldots,n_N}\frac{1}{(w+n_1a_1+\cdots+n_Na_N)^s}\int_0^\infty x^{s-1}e^{-x}\ dx\label{eq:zeta-condition}\\
    &=\sum_{n_1,\ldots,n_N}\int_0^\infty\frac{x^{s-1}e^{-x}}{(w+n_1a_1+\cdots+n_Na_N)^s}\ dx\\
    &=\sum_{n_1,\ldots,n_N}\int_0^\infty u^{s-1}e^{-(w+n_1a_1+\cdots+n_Na_N)u}\ du.
\end{align}
The sum can be moved inside by the dominated convergence theorem, giving us
\begin{align*}
    \zeta_N(s,w\lvert a_1,\ldots,a_N)\Gamma(s)&=\int_0^\infty u^{s-1}e^{-wu}\sum_{n_1,\ldots,n_N}e^{-(n_1a_1+\cdots+n_Na_N)u}\ du,
\end{align*}
and noting that $\Re(a_j)>0$, we may apply the geometric series formula so that
\begin{align*}
    \zeta_N(s,w\lvert a_1,\ldots,a_N)\Gamma(s)&=\int_0^\infty u^{s-1}e^{-wu}\frac{1}{1-e^{-a_1u}}\cdots\frac{1}{1-e^{-a_Nu}}\ du\\
    &=\int_0^\infty u^{s-1}e^{-wu}\prod_{j=1}^N(1-e^{-a_ju})^{-1}\ du.
\end{align*}
The alternating version is proved similarly.
\end{proof}

A consequence of this general result is the following Berndt-type integral identity.
\begin{corollary}\label{cor:berndt}
Let $\Re(s)>2$, $\Re(a)>0$, and $-\Re(a)<\Im(b)<\Re(a)$. Then we have
\begin{align*}
    \frac{1}{2\Gamma(s)}\int_0^\infty\frac{x^{s-1}}{\cosh(ax)-\cos(bx)}\ dx=\zeta_2(s,a\lvert a-bi,a+bi),
\end{align*}
and in particular,
\begin{align*}
\frac{1}{2\Gamma(s)}    \int_0^\infty\frac{x^{s-1}}{\cosh(x)-\cos(x)}\ dx=\zeta_2(s,1\lvert1-i,1+i).
\end{align*}
The alternating case is given by
\begin{align*}
    \frac{1}{2\Gamma(s)}\int_0^\infty\frac{x^{s-1}}{\cosh(ax)+\cos(bx)}\ dx=\widetilde{\zeta}_2(s,a\lvert a-bi,a+bi).
\end{align*}
\end{corollary}

Note that the integral on the left hand side of each identity appearing in Corollary~\ref{cor:berndt} is even with respect to the parameters $a$ and $b$. The Corollary can therefore be easily extended to $\Re(a)<0$. More generally, we may consider the Dirichlet-type multiple series
\begin{align*}
    L_\psi(s,w\lvert a_1,\ldots,a_N)=\sum_{n_1,\ldots,n_N\ge0}\frac{\psi(n_1,\ldots,n_N)}{(w+n_1a_1+\cdots+n_Na_N)^s},
\end{align*}
where $\psi$ is a multi-indexed complex sequence, and hope for a similar integral representation. Formally, we have
\begin{align*}
    L_\psi(s,w\lvert a_1,\ldots,a_N)\Gamma(s)&=\sum_{n_1,\ldots,n_N\ge0}\int_0^\infty\frac{x^{s-1}e^{-x}\psi(n_1,\ldots,n_N)}{(w+n_1a_1+\cdots+n_Na_N)^s}\ dx\\
    &=\sum_{n_1,\ldots,n_N\ge0}\int_0^\infty \psi(n_1,\ldots,n_N)x^{s-1}e^{-(w+n_1a_1+\cdots+n_Na_N)x}\ dx\\
    &=\int_0^\infty x^{s-1}\sum_{n_1,\ldots,n_N\ge0}\psi(n_1,\ldots,n_N)e^{-(w+n_1a_1+\cdots+n_Na_N)x}\ dx.
\end{align*}
The choice of $\psi=1$ recovers Proposition~\ref{prop:general}. Meanwhile, other choices of $\psi$ result in similar identities. If $\psi$ is separable in the sense that $\psi(n_1,\ldots,n_N)=\psi_1(n_1)\cdots\psi_N(n_N)$ for some $\psi_1,\ldots,\psi_N$, then we observe that $L_\psi$ obtains the particularly simple form
\begin{equation*} L_\psi(s,w|a_1,\ldots,a_N)=\frac{1}{\Gamma(s)}\int_0^\infty x^{s-1}e^{-wx}\prod_{k=1}^{N}\hat{\psi}_k(a_kx)dx,
\end{equation*}
with the Fourier transforms
\[
\hat{\psi}_k(x) = \sum_{n\ge0}\psi_k(n)e^{-n x},
\]
and we list some examples of this type in Table~\ref{tab:chis}. 

Let us turn to the Berndt-type integrals \eqref{eq:berndt-int} of arbitrary order and produce an evaluation using the Barnes zeta function and its alternating counterpart. In what follows, the notation $(a,b)^p$ indicates that the symbols $a,b$ are repeated $p$ times i.e. $(a,b)^3=a,b,a,b,a,b$.

\begin{table}[ht]
    \centering
    \begin{tabular}{||c|c|c||}
    \hline
        No. & $\psi(n_1,\ldots,n_N)$ & $L_\psi(s,w\lvert a_1,\ldots,a_N)$ \\
        \hline\hline
        1 & 1 & $\frac{1}{\Gamma(s)}\int_0^\infty x^{s-1}e^{-wx}\prod_{j=1}^N(1-e^{-a_j x})^{-1}dx$ \\
        \hline
        2 & $(-1)^{n_1+\cdots+n_N}$ & $\frac{1}{\Gamma(s)}\int_0^\infty x^{s-1}e^{-wx}\prod_{j=1}^N(1+e^{-a_j x})^{-1}dx$ \\
        \hline
        3 & $n_1\cdots n_N$ & $\frac{1}{\Gamma(s)}\int_0^\infty x^{s-1}e^{-wx}\prod_{j=1}^N(e^{a_j x/2}-e^{-a_j x/2})^{-2}dx$ \\
        \hline
        4 & $\frac{c_1^{n_1}\cdots c_N^{n_N}}{n_1!\cdots n_N!}$ & $\frac{1}{\Gamma(s)}\int_0^\infty x^{s-1}e^{-wx}\prod_{j=1}^Ne^{c_j e^{-a_j x}}dx$ \\
        \hline
        5 & $\sin(n_1)\cdots\sin(n_N)$ & $\frac{1}{\Gamma(s)}\int_0^\infty x^{s-1}e^{-wx}\prod_{j=1}^N\left(\frac{-i(e^{2i}-1)}{2}\right)(e^{a_j x+i} -e^{-(a_j x+i)})^{-1}dx$ \\
        \hline
        6 & $\frac{1}{\Gamma(1+\frac{n_1}{2})\cdots\Gamma(1+\frac{n_N}{2})}$ & $\frac{1}{\Gamma(s)}\int_0^\infty x^{s-1}e^{-wx}\prod_{j=1}^Ne^{e^{-2a_jx}}(1+\text{Erf}(e^{-a_jx}))dx$ \\
        \hline
        7 & $H_{n_1}\cdots H_{n_N}$ & $\frac{1}{\Gamma(s)}\int_0^\infty x^{s-1}e^{-wx}(-1)^N\prod_{j=1}^N\log(1-e^{-a_jx})(1-e^{-a_jx})^{-1}dx$ \\
        \hline
        8 & $\frac{(-1)^{n_1+\cdots+n_N}}{(n_1+1)\cdots(n_N+1)}$ & $\frac{1}{\Gamma(s)}\int_0^\infty x^{s-1}e^{-wx}\prod_{j=1}^Ne^{a_jx}\log(1+e^{-a_jx})dx$ \\
        \hline
        9 & $(-1)^{n_1+\cdots+n_N}n_1\cdots n_N$ & $\frac{1}{\Gamma(s)}\int_0^\infty x^{s-1}e^{-wx}(-1)^N\prod_{j=1}^N(e^{a_j x/2}+e^{-a_j x/2})^{-2}dx$ \\
        \hline
        10 & $\sqrt{n_1\cdots n_N}$ & $\frac{1}{\Gamma(s)}\int_0^\infty x^{s-1}e^{-wx}\prod_{j=1}^N\text{polylog}(-\frac12,e^{-a_j x})dx$ \\
        \hline
        11 & $(-1)^{n_1+\cdots+n_N}\sqrt{n_1\cdots n_N}$ & $\frac{1}{\Gamma(s)}\int_0^\infty x^{s-1}e^{-wx}\prod_{j=1}^N\text{polylog}(-\frac12,\sinh(a_j x)-\cosh(a_jx))dx$ \\
        \hline
    \end{tabular}
    \caption{List of Identities given by a choice of $\psi$, where $H_n$ denotes the $n$-th harmonic number. Here we assume that $\Re(s)>N$, $\Re(a_j)>0$, and that $\Re(w)>0$.}
    \label{tab:chis}
\end{table}

\begin{proposition}\label{prop:berndt-evaluation}
Let $p>0$ be an integer and let $s>2p$ be a real number. Then
\begin{equation*}
\int_0^\infty\frac{x^{s-1}}{(\cosh(x)-\cos(x))^p}dx=2^p\Gamma(s)\zeta_{2p}(s,p|(1+i,1-i)^p)
\end{equation*}
and
\begin{equation*}
\int_0^\infty\frac{x^{s-1}}{(\cosh(x)+\cos(x))^p}dx=2^p\Gamma(s)\Tilde{\zeta}_{2p}(s,p|(1+i,1-i)^p)
\end{equation*}
\end{proposition}
\begin{proof}
We will tackle the case of arbitrary products of sinh and cosh in the denominator and then specialize to the case of \eqref{eq:berndt-int}. Observe that
\begin{align*}
    I&:=\frac{1}{\Gamma(s)}\int_0^\infty \frac{x^{s-1}e^{-wx}}{\sinh(a_1x)\cdots\sinh(a_Mx)\cosh(b_1x)\cdots\cosh(b_Nx)}dx\\
    &=\frac{1}{\Gamma(s)}\int_0^\infty \frac{x^{s-1}e^{-wx}}{\left(\prod_{i=1}^M\frac{e^{a_i x}-e^{-a_ix}}{2}\right)\left(\prod_{j=1}^N\frac{e^{b_j x}+e^{-b_jx}}{2}\right)}dx\\
    &=\frac{2^{M+N}}{\Gamma(s)}\int_0^\infty \frac{x^{s-1}e^{-wx}}{\left(\prod_{i=1}^M(e^{a_i x}-e^{-a_ix})\right)\left(\prod_{j=1}^N(e^{b_j x}+e^{-b_jx})\right)}dx\\
    &=\frac{2^{M+N}}{\Gamma(s)}\int_0^\infty \frac{x^{s-1}e^{-x(w+a_1+\cdots+a_M+b_1+\cdots+b_N)}}{\left(\prod_{i=1}^M(1-e^{-2a_ix})\right)\left(\prod_{j=1}^N(1+e^{-2b_jx})\right)}dx\\
    &=\frac{2^{M+N}}{\Gamma(s)}\int_0^\infty x^{s-1}e^{-x(w+\sum_{i=1}^Ma_i+\sum_{i=1}^Nb_i)}\sum_{n_1,\ldots,n_M,k_1,\ldots,k_N\ge0}(-1)^{k_1+\cdots+k_N}e^{-2x(\sum_{i=1}^Ma_in_i+\sum_{i=1}^Nb_ik_i)}dx\\
    &=\frac{2^{M+N}}{\Gamma(s)}\sum_{n_1,\ldots,n_M,k_1,\ldots,k_N\ge0}\int_0^\infty (-1)^{k_1+\cdots+k_N}x^{s-1}e^{-x(w+\sum_{i=1}^Ma_i+\sum_{i=1}^Nb_i+2(\sum_{i=1}^Ma_in_i+\sum_{i=1}^Nb_ik_i))}dx\\
    &=\frac{2^{M+N}}{\Gamma(s)}\sum_{n_1,\ldots,n_M,k_1,\ldots,k_N\ge0}\int_0^\infty \frac{(-1)^{k_1+\cdots+k_N}x^{s-1}e^{-x}dx}{(w+\sum_{i=1}^Ma_i+\sum_{i=1}^Nb_i+2(\sum_{i=1}^Ma_in_i+\sum_{i=1}^Nb_ik_i))^s}\\
    &=2^{M+N}\sum_{n_1,\ldots,n_M,k_1,\ldots,k_N\ge0} \frac{(-1)^{k_1+\cdots+k_N}dx}{(w+\sum_{i=1}^Ma_i+\sum_{i=1}^Nb_i+2(\sum_{i=1}^Ma_in_i+\sum_{i=1}^Nb_ik_i))^s}.
\end{align*}
Now writing $\cosh(x)-\cos(x)=2\sinh\left(\frac{1-i}{2}x\right)\sinh\left(\frac{1+i}{2}x\right)$ and $\cosh(x)+\cos(x)=2\cosh\left(\frac{1-i}{2}x\right)\cosh\left(\frac{1+i}{2}x\right)$, the proposition follows.
\end{proof}
\begin{remark}
In view of the previous proposition, it is worth pointing out that $\zeta_{2p}(s,p|(a,b)^{p})$ can be simplified to a double sum for any integer value of $p\ge 2,$ as a Dirichlet analog of $\zeta_2$. Indeed, noticing that
\begin{align*}
    \sum_{n_1,\ldots,n_p,m_1,\ldots,m_p\ge0}\frac{1}{(w+a(n_1+\cdots+n_p)+b(m_1+\cdots+m_p))^s}=\sum_{n,m\ge0}\frac{P(n)P(m)}{(w+an+bm)^s},
\end{align*}
where $P(n)=\lvert\{n_1\ge0,\ldots,n_p\ge0 : n_1+\cdots+n_p=n\}\vert = \binom{n+p-1}{n}$ is a counting function, it follows that
\begin{equation*}
    \zeta_{2p}(s,w|(a,b)^{p})=L_\psi(s,w|a,b) = \sum_{n,m\ge0}\frac{\psi(n,m)}{(w+an+bm)^s}
\end{equation*}
with $\psi(n,m)=P(n)P(m)= \binom{n+p-1}{n} \binom{m+p-1}{m}$. This result extends to the more general case, assuming that $a_i,b_j\in \mathbb{Z}$
\[
 \sum_{n_1,\ldots,n_p,m_1,\ldots,m_p\ge0}\frac{1}{(w+a_1n_1+\cdots+a_pn_p+b_1m_1+\cdots+b_pm_p)^s}=\sum_{n,m\ge0}\frac{\psi_{a,b}(n,m)}{(w+an+bm)^s}
\]
with $\psi_{a,b}(n,m)=\psi_{a}(n)\psi_{b}(m)$ and with the partition function  $\psi_{a}(n)=\lvert\{n_1\ge0,\ldots,n_p\ge0 : n_1a_1+\cdots+n_pa_p=n\}\vert.$ 
\end{remark}
In the work of Xu and Zhao, explicit evaluations of the quadratic Berndt-type integrals are given as rational linear combinations of products of $\Gamma(1/4)$ and $\pi$. For example, they show that
\begin{equation*}
\int_0^\infty\frac{x^5}{(\cosh(x)-\cos(x))^2}dx=\frac{\Gamma^{16}(1/4)}{3\cdot2^{14}\pi^6}-\frac{\Gamma^8(1/4)}{2^8\pi^2},
\end{equation*}
and so it follows from Proposition~\ref{prop:berndt-evaluation} that
\begin{equation*}
    \zeta(6,2|(1+i,1-i)^2)=\frac{1}{480}\left(\frac{\Gamma^{16}(1/4)}{3\cdot2^{14}\pi^6}-\frac{\Gamma^8(1/4)}{2^8\pi^2}\right).
\end{equation*}
In this way, the results of Xu and Zhao, together with the above results, reveal explicit evaluations of the Barnes zeta function in terms of the gamma function. More generally, they show that for all integers $p\ge1$ and $s\ge\lceil m/2\rceil$, the Berndt-type integrals satisfy
\begin{equation*}
    I_+(4s+2,p)\in\mathbb{Q}\left[\Gamma^4(1/4),\pi^{-1}\right],
\end{equation*}
where $\mathbb{Q}[x,y]$ denotes all rational linear combinations of products of $x$ and $y$ (the polynomial ring generated by $x,y$ with rational coefficients). Moreover, they show that the degrees of $\Gamma^4(1/4)$ have the same parity as $p$ and are between $2s-p+2$ and $2s+p$, inclusive, while the degrees of $\pi^{-1}$ are between $2s-p+2$ and $2s+3p-2$, inclusive. Similarly, they show that
\begin{equation*}
    I_-(4s,2p+1)\in\mathbb{Q}\left[\Gamma^4(1/4),\pi^{-1}\right]
\end{equation*}
for all $s\ge k+1\ge1$ and that the degrees of $\Gamma^4(1/4)$ in each term are always even and between $2s-2p$ and $2s+2p$, while the degrees of $\pi^{-1}$ are between $2s-2p$ and $2s+6p$. In the even order case, they show that
\begin{equation*}
    I_-(4s+1,2p)\in\mathbb{Q}\left[\Gamma^4(1/4),\pi^{-1}\right]
\end{equation*}
for all integers $s\ge k\ge1$ and the powers of $\Gamma^4(1/4)$ are always even and between $2s+2-2p$ and $2s+2p$, while the degrees of $\pi^{-1}$ are between $2s+2-2p$ and $2s+6p-2$. The following corollary therefore follows from Proposition~\ref{prop:berndt-evaluation}.
\begin{corollary}
    Let $p\ge1$ and $s\ge\lceil p/2\rceil$ be integers. Then
    \begin{equation*}
        \Tilde{\zeta}_{2p}(4s+2,p|(1+i,1-i)^p)\in\mathbb{Q}\left[\Gamma^4(1/4),\pi^{-1}\right].
    \end{equation*}
    Similarly, if $s\ge p+1\ge1$, then
    \begin{equation*}
        \zeta_{4p+2}(4s,2p+1|(1+i,1-i)^{2p+1})\in\mathbb{Q}\left[\Gamma^4(1/4),\pi^{-1}\right],
    \end{equation*}
    and if $s\ge p\ge1$, then
    \begin{equation*}
        \zeta_{4p}(4s+1,2p|(1+i,1-i)^{2p})\in\mathbb{Q}\left[\Gamma^4(1/4),\pi^{-1}\right].
    \end{equation*}
\end{corollary}

We will close this section by pointing out two consequences of the connection between Berndt-type integrals and the Barnes zeta function.
\subsection{Laplace transforms}
The Barnes zeta representation lends itself to an interesting interpretation in terms of the Laplace transform. Indeed, we have the following proposition showing that the integration of an analytic function $f$ against the kernel $(\cosh{ax}-\cos{bx})^{-1}$ performs a lattice summation of the Laplace transform of $f.$

\begin{proposition}
Let $\Re(a)>0$ and $-\Re(a)<\Im(b)<\Re(a)$ and let $f(x)=\sum_{n\ge2}\frac{c_{n}}{n!}x^{n}$ be an entire function with $\lvert c_n\rvert<k^n$ for some constant $k<\Re(a)$. Let $F(p)$ denote the Laplace transform of $f$. Then
\begin{equation}\label{eq:berndt-laplace-transform-identity}
    \int_{0}^{\infty}\frac{f\left(x\right)}{\cosh ax-\cos bx}dx  =2\sum_{p,q\ge0}F\left(a+p\left(a-i b\right)+q\left(a+i b\right) \right)
\end{equation}
whenever the lattice sum converges.
\end{proposition}
\begin{proof}
The Laplace transform of $f$ exists for $\Re(s)>k$ and is given by
\[
F\left(s\right)=\sum_{n\ge2}\frac{c_{n}}{s^{n+1}}.
\]
Consider now the integral
\begin{equation*}
    \int_{0}^{\infty}\frac{f\left(x\right)}{\cosh ax-\cos bx}dx.
\end{equation*}
We would like to expand $f$ and swap the order of summation and integration. Since $f(0)=f'(0)=0$, we may write $f(x)=x^2g(x)$ for some entire function $g$. Moreover, $\cosh(ax)$ is bounded below by 1 and strictly increasing for all $x\in(0,\infty)$, and we have $\cosh(ax)=1$ if and only if $x=0$; thus, the only singularity for $1/\lvert\cosh(ax)-\cos(bx)\rvert$ is at $x=0$. Observe that
\begin{equation*}
    \lim_{x\to0^+}\left\lvert\frac{x^2}{\cosh(ax)-\cos(bx)}\right\rvert=\lim_{x\to0^+}\frac{x^2}{\cosh(ax)-\cos(bx)}=\frac{2}{a^2+b^2}.
\end{equation*}
Now, we have
\begin{equation*}
    \int_0^\infty\sum_{n=2}^\infty \left\lvert\frac{ c_nx^n}{n!(\cosh(ax)-\cos(bx))}\right\rvert dx<\int_0^\infty \frac{e^{kx}-1-kx}{\cosh(ax)-\cos(bx)}dx.
\end{equation*}
Since $k<\Re(a)$, the integral converges. It therefore follows from the Fubini-Tonelli theorem that the order of summation and integration can be swapped. Thus,
\begin{align*}
\int_{0}^{\infty}\frac{f\left(x\right)}{\cosh ax-\cos bx}dx & =\sum_{n\ge2}\frac{c_{n}}{n!}\int_{0}^{\infty}\frac{x^{n}}{\cosh ax-\cos bx}dx\\
 & =2\sum_{n\ge2}\frac{c_{n}}{n!}\Gamma\left(n+1\right)\zeta_{2}\left(n+1,a,\vert a-i b,a+i b\right)\\
 & = 2\sum_{n\ge2}c_n\sum_{p,q\ge0}\left(a+p\left(a-i b\right)+q\left(a+i b\right) \right)^{-n-1}\\
 & = 2\sum_{p,q\ge0}F\left(a+p\left(a-i b\right)+q\left(a+i b\right) \right).
\end{align*}
where in the second to last equality the swapping of the sums is justified by the observation that $\lvert a+p\left(a-i b \right)+q\left(a+i b\right)\rvert^{-n-1}<\lvert a\rvert^{-n-1}$, so that the assumption of $\lvert c_n\rvert<k^n<\lvert a\rvert^n$ guarantees absolute convergence of the sum over $n$. Thus, the Fubini-Tonelli theorem applies whenever the resulting lattice sum of the Laplace transform converges.
\end{proof}

This result extends to the case of the $(\cosh{ax}+\cos{bx})^{-1}$ denominator as follows.

\begin{proposition}\label{prop:laplace}
Let $\Re(a)>0$ and $-\Re(a)<\Im(b)<\Re(a)$ and let $f(x)=\sum_{n\ge0}\frac{c_{n}}{n!}x^{n}$ be an entire function with $\lvert c_n\rvert<k^n$ for some constant $k<\Re(a)$. Let $F(p)$ denote the Laplace transform of $f$. Then
\begin{equation}\label{eq:berndt-laplace-transform-identity2}
    \int_{0}^{\infty}\frac{f\left(x\right)}{\cosh ax+\cos bx}dx  =
    2\sum_{p,q\ge0}(-1)^{p+q}F\left(a+p\left(a-i b \right)+q\left(a+i b \right) \right).
\end{equation}
whenever the lattice sum converges.
\end{proposition}

\begin{proof}
The Laplace transform of $f$ exists for $\Re(s)>k$ and is given by
\[
F\left(s\right)=\sum_{n\ge0}\frac{c_{n}}{s^{n+1}},
\]
from which we deduce
\[
\int_{0}^{\infty}\frac{f\left(x\right)}{\cosh ax+\cos bx}dx=\sum_{n\ge0}\frac{c_{n}}{n!}\int_{0}^{\infty}\frac{x^{n}}{\cosh ax+\cos bx}dx
\]
\[
=\sum_{n\ge0}\frac{c_{n}}{n!}2\Gamma\left(n+1\right)\tilde{\zeta}_2\left(n+1,a\vert a-i b,a+i b\right)
\]
\[
=2\sum_{n\ge0}c_{n}\sum_{p,q\ge0}\frac{\left(-1\right)^{p+q}}{\left(a+p\left(a-i b \right)+q\left(a+i b\right)\right)^{n+1}}
\]
\[
=2\sum_{p,q\ge0}\left(-1\right)^{p+q}F\left(a+p\left(a-i b\right)+q\left(a+i b\right)\right),
\]
where in the first equality, the swapping of the summation and integration is justified by the inequality $\lvert c_n\rvert<k^n<(\Re(a))^n$, which ensures that
\begin{equation*}
    \int_0^\infty\sum_{n=0}^\infty\left\lvert\frac{c_nx^n}{n!(\cosh(ax)+\cos(bx))}\right\rvert dx<\int_0^\infty \frac{e^{kx}}{\cosh(ax)+\cos(bx)}dx<\infty
\end{equation*}
so that the Fubini-Tonelli theorem can be applied. To justify the swapping of the summations in the second to last equality, we observe that $\lvert a+p\left(a-i b \right)+q\left(a+i b\right)\rvert^{-n-1}<\lvert a\rvert^{-n-1}$, so that the assumption of $\lvert c_n\rvert<k^n<\lvert a\rvert^n$ guarantees absolute convergence of the sum over $n$. Thus, the Fubini-Tonelli theorem applies whenever the resulting lattice sum of the Laplace transform converges.
\end{proof}

\begin{corollary}
The special case $f\left(x\right)=\frac{\sin x}{x}$ produces the identity
\[
\sideset{}{'}\sum\left(-1\right)^{p+q+1}\arctan\left(\frac{p+q+1}{p\left(p+1\right)+q\left(q+1\right)}\right)=\frac{\pi}{4}
\]
where the $\sideset{}{'}\sum$  sign indicates summation over the set of integers $\left\{p\ge 0,q\ge0, (p,q)\ne (0,0) \right\}.$
\end{corollary}
\begin{proof} The Laplace transform of  $f$ is given by $F\left(s\right)=\arctan\left(\frac{1}{s}\right)$, from which it follows that
\[
F\left(1+p+q+i\left(p-q\right)\right)=\arctan\left(\frac{1}{1+p+q+i\left(p-q\right)}\right).
\]
We need only the real part of the Laplace transform. Denote
\[
z=\frac{1}{1+p+q+i\left(p-q\right)}=x+i y.
\]
Using \cite[4.23.36]{nist}, 
\[
\Re\left(\arctan\left(x+i y\right)\right)=\frac{1}{2}\arctan\left(\frac{2x}{1-x^{2}-y^{2}}\right).
\]
Then with $x=\frac{1+p+q}{\left(1+p+q\right)^{2}+\left(p-q\right)^{2}}$
and $y=-\frac{p-q}{\left(1+p+q\right)^{2}+\left(p-q\right)^{2}}$, it follows that
\begin{align*}
\Re\left( F\left(1+p\left(1+i\right)+q\left(1-i\right)\right)\right)&=\Re\left(\arctan\left(\frac{1}{1+p+q+i\left(p-q\right)}\right)\right)\\
&=\frac{1}{2}\arctan\left(\frac{p+q+1}{p\left(p+1\right)+q\left(q+1\right)}\right)
  \end{align*}
whereas, for $\left(p,q\right)=\left(0,0\right),$
\[
F\left(1\right)=\arctan\left(1\right)=\frac{\pi}{4}.
\]
To invoke Proposition~\ref{prop:laplace}, we must show that there is a $k<1$ such that $n!/(2n+1)!<k^n$ for $n>0$ (indeed, the inequality will fail for $n=0$, but we can subtract this part and integrate it separately). Taking the logarithm produces $-\sum_{j=n+1}^{2n+1}\log(j)<n\log(k)$, from which it follows that $\log(k)>-\frac{1}{n}(2n+1-n-1)\log(2n+1)=-\log(2n+1)$. Now, exponentiating, we see that the inequality holds whenever $k>\frac{1}{2n+1}$ for all $n>0$. It therefore suffices to take $k>\frac{1}{3}$. Applying Proposition~\ref{prop:laplace} now produces
\begin{align*}
\frac{\pi}{4}&=\int_{0}^{\infty}\frac{\sin x}{x\left(\cosh x+\cos x\right)}dx\\
&=2F\left(1\right)+2\sideset{}{'}\sum F\left(1+p\left(1+i\right)+q\left(1-i\right)\right)\\
&=\frac{\pi}{2}+2\sideset{}{'}\sum\left(-1\right)^{p+q}\frac{1}{2}\arctan\left(\frac{p+q+1}{p\left(p+1\right)+q\left(q+1\right)}\right),
\end{align*}
and it follows that
\[
\sideset{}{'}\sum\left(-1\right)^{p+q+1}\arctan\left(\frac{p+q+1}{p\left(p+1\right)+q\left(q+1\right)}\right)=\frac{\pi}{4}.
\]
\end{proof}
\subsection{Analytic continuation}
The second application of the Barnes zeta representation is an analytic continuation of a normalized version of a Berndt-type integral: since at $x=0,$
\[
\frac{x^{s-1}}{\left(\cosh x-\cos x\right)^{p}}\sim x^{s-1-2p}
\]
the normalized integral
\[
\hat{I}\left(s\right)=\frac{1}{\Gamma(s)}I(s)=\frac{1}{\Gamma\left(s\right)}\int_{0}^{\infty}\frac{x^{s-1}}{\left(\cosh x-\cos x\right)^{p}}dx
\]
is convergent in the half-plane $\Re s>2p.$ Based on a result by Ruijsenaars', it can be analytically continued as follows.
\begin{proposition}
    The representation 
    \[
\hat{I}\left(s\right)=\frac{1}{\Gamma\left(s\right)}\int_{0}^{\infty}\frac{x^{s-1}}{\left(\cosh x-\cos x\right)^{p}}dx=2^p\zeta_{2p}(s,p|(1+i,1-i)^p)
\]
produces an analytic continuation of $\hat{I}\left(s\right)$ to the whole complex plane except at the points $s=0,1,\dots,2p$ with $p$ a positive integer. 
Moreover, for positive $s>2p$ , 
\[
\hat{I}\left(-s\right)=\frac{(-2)^p}{(s+1)\dots (s+2p)}B_{2p+s}(p\vert(1+i,1-i)^p)
\]with the Bernoulli-Barnes polynomials $B_n(z,\mathbf{a})$ defined by the generating function
\[
\sum_{n \ge 0} \frac{B_n(z,(a_1,\dots,a_p))}{n!}u^n = e^{zu}\prod_{i=1}^{p}\frac{a_i u}{e^{a_i u}-1}.
\]
\end{proposition}
\begin{proof}
The multivariate integral representation in Ruijsenaars' article \cite[(4.13)]{ruijsenaars2000} 
\[
\zeta_{2p}\left(s,w\vert\mathbf{a}\right)=\frac{1}{\left(s-1\right)\dots\left(s-2p\right)}\int_{\mathbb{R}^{n}}\left(w+\sum_{i=1}^{2p}(i x_j-\frac{a_j}{2})\right)^{2p-s}\prod_{i=1}^{2p}\phi\left(\frac{x_{i}}{a_{i}}\right)dx_{i}
\]
with $\phi\left(x\right)=\frac{\pi}{2}\textnormal{sech}^{2}\left(\pi x\right)$ produces an analytic continuation (in the variable $s$) of the Barnes zeta function  to the whole complex plane except at the points $s=0,1,\dots,2p.$    An integral representation for the Bernoulli Barnes polynomials, given by Ruijsenaars as 
\[B_n(z,(a_1,\dots,a_{2p})) = \int_{\mathbb{R}^{n}}\left(w+\sum_{j=1}^{2p}(i x_j-\frac{a_j}{2})\right)^{n}\prod_{i=1}^{2p}\phi\left(\frac{x_{i}}{a_{i}}\right)dx_{i},\] produces the desired result.
\end{proof}


\section{Kuznetsov's Approach}\label{sec:jacobi}
In \cite{kuznetsov}, Kuznetsov gives a direct evaluation of a curious integral appearing in the work of Ismail and Valent \cite{ismail1998} which takes the form
\begin{equation*}
    \int_{\mathbb{R}}\frac{dx}{\cos(\sqrt{x}K)+\cosh(\sqrt{x}K')}=2,
\end{equation*}
where $K=K(k)$ denotes the complete elliptic integral of the first kind with elliptic modulus $k\in(0,1)$, and $K'=K(k')$, where $k'=\sqrt{1-k^2}$ is the complementary elliptic modulus. 

In fact, Kuznetsov shows more generally that for $u\in\mathbb{C}$ satisfying $\lvert\Re(u)\rvert<K$ and $\lvert\Im(u)\rvert<K'$, we have the generating function
\begin{equation}\label{eq:kuznetsov+}
    \int_\mathbb{R}\frac{\sin(\sqrt{x}u)}{\sqrt{x}}\frac{dx}{\cos(\sqrt{x}K)+\cosh(\sqrt{x}K')}=2\frac{\text{sn}(u,k)}{\text{cd}(u,k)}=2 \widetilde{\text{nc}}(u,k),
\end{equation}
where sn, nc and cd are Jacobi elliptic functions and $\widetilde{\text{nc}}(u,k) = \frac{d}{du}\log{\text{nc}}(u,k)$. This identity has several implications which we will now discuss. 

\subsection{Connection to Berndt Integrals} Let us define the moment integrals (or Kuznetsov's integrals)
\begin{equation}
\label{eq:In+}
    I_{n}^{+}=\frac12\int_\mathbb{R}\frac{x^ndx}{\cos(\sqrt{x}K)+\cosh(\sqrt{x}K')},
\end{equation}
with $n$ a positive integer.
Note that the integrand in \eqref{eq:kuznetsov+} is in $L_1(0,\infty)$ as a function of $x$ for all $u$, as are its partial derivatives with respect to $u$. Then by differentiating \eqref{eq:kuznetsov+} under the integral sign with respect to $u$ using Leibniz integral rule \cite{folland1999real}, it follows that 
\begin{equation}\label{eq:kuznetsov-result}
    \frac12\int_\mathbb{R}\frac{x^ndx}{\cos(\sqrt{x}K)+\cosh(\sqrt{x}K')}
    =(-1)^n\frac{d^{2n+1}}{du^{2n+1}}\widetilde{\text{nc}}(u,k)\bigg\vert_{u=0}.
\end{equation}
Clearly the moment integrals $I_{n}^{+}$ are closely related to the Berndt-type integrals defined in \eqref{eq:berndt-int}; indeed, a substitution $z=\sqrt{x}$ in \eqref{eq:In+}  reveals that
\begin{equation}
\label{eq:Kuznetsov generating function}
    \int_L\frac{z^{2n+1}dz}{\cos(zK)+\cosh(zK')}
    =(-1)^n\frac{d^{2n+1}}{du^{2n+1}}\widetilde{\text{nc}}(u,k)\bigg\vert_{u=0},
\end{equation}
where $L$ is the contour which travels from infinity down the positive imaginary axis until reaching zero, where the contour then begins traveling along the orthogonal positive real axis to infinity. 
In the lemniscatic case, the modulus $\tilde{k}=\frac{1}{\sqrt2}$ is equal to the complementary modulus and the elliptic integral becomes $\Tilde{K}=\frac{\pi^{\frac{3}{2}}}{2\Gamma^2(\frac{3}{4})}$. Thus, by making an additional substitution, we have
\begin{equation*}
    \int_L\frac{z^{2n+1}dz}{\cos(z)+\cosh(z)}=
    (-1)^n\Tilde{K}^{2n+2}\frac{d^{2n+1}}{du^{2n+1}}\widetilde{\text{nc}}(u,\tilde{k})\bigg\vert_{u=0},
\end{equation*}
and now substituting $x=-iz$ along the vertical portion of the contour finally produces
\begin{equation*}
    (1+(-1)^n)\int_0^\infty\frac{x^{2n+1}dx}{\cos(x)+\cosh(x)}
    = \Tilde{K}^{2n+2}\frac{d^{2n+1}}{du^{2n+1}}\widetilde{\text{nc}}(u,\tilde{k})\bigg\vert_{u=0},
\end{equation*}
so that if $n=2m$ is even, we obtain a Berndt-type integral representation in terms of Jacobi elliptic functions
\begin{equation*}
    \int_0^\infty\frac{x^{4m+1}dx}{\cos(x)+\cosh(x)}=\frac12\left(\frac{\pi^{\frac{3}{2}}}{2\Gamma^2(\frac{3}{4})}\right)^{4m+2}
    \frac{d^{4m+1}}{du^{4m+1}}\widetilde{\text{nc}}(u,1/\sqrt{2})\bigg\vert_{u=0}.
\end{equation*}
The more general  identity \eqref{eq:Kuznetsov generating function}  is used in the next subsections to produce equivalent forms of Berndt's integrals in terms of special functions.
\subsection{Moment Polynomials of Lomont and Brillhart}
Evaluation of the first values of Kuznetsov's integral 
\[
I_0^{+} = 1,\,\,I_1^{+}=-2(1-2k^2),\,\,
I_2^{+}=16(k^4-k^2+1)
\]
suggests that they can be expressed as a polynomial function of the elliptic modulus $k.$ The work of Lomont and Brillhart \cite{Lomont} allows us to identify these polynomials and some of their properties: for example, this connection can be leveraged to derive recurrence relations between Berndt-type integrals. 

From \cite[(5.35)]{Lomont}, we have
\begin{equation}\label{eq:lomont-identity}
\log{\text{nc}\left(u,k\right)}=\sum_{n\ge0}2^{n}P_{n}\left(1-2k^{2},4k^{4}-4k^{2}+4\right)\frac{u^{2n+2}}{\left(2n+2\right)!}
\end{equation}
with $P_{n}\left(x,y\right)$ the moment polynomials defined in \cite[Ch.4]{Lomont}
by the recurrence
\[
P_{n+2}=xP_{n+1}+\left(y-x^{2}\right)\sum_{j=0}^{n}\binom{2n+2}{2j}P_{j}\sum_{l=0}^{n-j}\binom{2n-2j+1}{2l}P_{l}P_{n-j-l}.
\]
with initial values $P_0(x,y)=1,\,\,P_1(x,y)=x.$
These polynomials have integer coefficients and satisfy
\[
\deg_{x}P_{n}\left(x,y\right)=\begin{cases}
n & n\not\equiv2\mod3\\
3n & n\equiv2\mod3
\end{cases}
\]
and
\[
\deg_{y}P_{n}\left(x,y\right)=\left[\frac{n}{2}\right].
\]
First cases are \cite[Table 4.1]{Lomont}
\[
P_0(x,y)=1,\,\,P_1(x,y)=x,\,\,P_2(x,y)=y,\,\,P_3(x,y)=-10x^3+11xy.
\]
It follows from \eqref{eq:lomont-identity} that
\[
\widetilde{\textnormal{nc}}(u,k)=\frac{d}{du}\log \textnormal{nc}(u,k)=\sum_{n\ge0}2^{n}P_{n}\left(1-2k^{2},4k^{4}-4k^{2}+4\right)\frac{u^{2n+1}}{\left(2n+1\right)!},
\]
producing the following evaluation of the integral considered by Kuznetsov.

\begin{proposition}
Kuznetsov's integral satisfies
\begin{equation}
\label{eq:Kuznetsov polynomial}
I_{n}^{+}=\frac{1}{2}\int_{\mathbb{R}}\frac{x^{n}dx}{\cos\left(K\sqrt{x}\right)+\cosh\left(K'\sqrt{x}\right)}=\left(-1\right)^{n}2^{n}P_{n}\left(1-2k^{2},4k^{4}-4k^{2}+4\right),
\end{equation}
with $n$ a positive integer.
\end{proposition}
Identity \eqref{eq:Kuznetsov polynomial} induces the following properties:
\begin{corollary}
If $k^{2}$ is a rational number, then Kuznetsov's integral $I_{n}^{+}$ is a rational number. Moreover, since 
\[
I_{1}^{+}=-2\left(1-2k^{2}\right),\thinspace\thinspace I_{2}^{+}=16\left(1-k^{2}+k^{4}\right),
\]
the value of $I_n^{+}$ is a polynomial function of the initial values $I_1^{+}$ and $I_2^{+}$ given by
\[
I_{n}^{+}=\left(-2\right)^{n}P_{n}\left(-\frac{1}{2}I_{1}^{+},\frac{1}{4}I_{2}^{+}\right).
\]    
\end{corollary}
It is instructive to list several of the first examples for the lemniscatic case $k=\frac{1}{\sqrt{2}}$ \footnote{the sequence $\left(1,12,3024,\dots \right)$ appears as OEIS A104203 and coincides with the Taylor coefficients of the sine lemniscate function $sl(u,k)$}:
\[
I_{0}^{+}=\frac{1}{2}\int_{\mathbb{R}}\frac{dx}{\cos\left(K\sqrt{x}\right)+\cosh\left(K'\sqrt{x}\right)}=1
\]

\[
I_{1}^{+}=\frac{1}{2}\int_{\mathbb{R}}\frac{xdx}{\cos\left(K\sqrt{x}\right)+\cosh\left(K'\sqrt{x}\right)}=0
\]

\[
I_{2}^{+}=\frac{1}{2}\int_{\mathbb{R}}\frac{x^{2}dx}{\cos\left(K\sqrt{x}\right)+\cosh\left(K'\sqrt{x}\right)}=12
\]

\[
I_{3}^{+}=\frac{1}{2}\int_{\mathbb{R}}\frac{x^{3}dx}{\cos\left(K\sqrt{x}\right)+\cosh\left(K'\sqrt{x}\right)}=0
\]

\[
I_{4}^{+}=\frac{1}{2}\int_{\mathbb{R}}\frac{x^{4}dx}{\cos\left(K\sqrt{x}\right)+\cosh\left(K'\sqrt{x}\right)}=3024.
\]
In this case, $x=0,\,\,y=3$ and $P_{2n+1}\left(x,y\right)=P_{2n+1}\left(0,3\right)=0$
for all $n\ge1,$ as a consequence of the fact that $x$ always factors
in $P_{2n+1}(x,y)$, as can be seen from the recurrence relations satisfied by the polynomials $P_n$. Indeed, we have \cite[(4.50)]{Lomont}
\[
P_{n+2}=-n\left(2n+3\right)xP_{n+1}+\sum_{j=0}^{n}\left[2\binom{2n+3}{2j+3}-\binom{2n+4}{2j+4}\right]P_{j+2}P_{n-j}
\]
and
\[
P_{n+2}=xP_{n+1}+\left(y-x^{2}\right)\sum_{j=0}^{n}\sum_{l=0}^{n-j}\binom{2n+2}{2j}\binom{2n-2j+1}{2l}P_{j}P_{l}P_{n-j-l},
\]
from which we deduce the following proposition.
\begin{proposition}
The integrals $I_n^{+}$ satisfy the recurrence relations
\begin{equation*}
I_{n+2}^{+}=2n\left(2n+3\right)\left(1-2k^{2}\right)I_{n+1}^{+}+\sum_{j=0}^{n}\left[2\binom{2n+3}{2j+3}-\binom{2n+4}{2j+4}\right]I_{j+2}^{+}I_{n-j}^{+}
\end{equation*}
and
\begin{equation*}
I_{n+2}^{+}=-2\left(1-2k^{2}\right)I_{n+1}^{+}+12\sum_{j=0}^{n}\binom{2n+2}{2j}I_{j}^{+}\sum_{l=0}^{n-j}\binom{2n-2j+1}{2l}I_{l}^{+}I_{n-j-l}^{+}.
\end{equation*}
\end{proposition}

\subsection{Lambert Series Representation}
Kuznetsov's result also produces a Lambert series representation for Kuznetsov's integrals. 
\begin{proposition}
A Lambert series representation for Kuznetsov's integrals is
\[
\frac{1}{2}\int_{\mathbb{R}}\frac{x^{p-1}dx}{\cos\left(K\sqrt{x}\right)+\cosh\left(K'\sqrt{x}\right)}=\left(\frac{\pi}{K}\right)^{2p}\left[-\frac{1}{2}E_{2p-1}\left(0\right)+2\sum_{n\ge1}\frac{n^{2p-1}q^{n}}{1+\left(-q\right)^{n}}\right].
\]    
\end{proposition}
\begin{proof}
Indeed, by \cite[vol.3 p.15]{tolke3},
\[
\log \textnormal{cn}\left(u,k\right)=\log\cos\left(\frac{\pi u}{2K}\right)-4\sum_{n\ge1}q^{n}\frac{\sin^{2}\left(\frac{n\pi u}{2K}\right)}{n\left(1+\left(-q\right)^{n}\right)}.
\]
We expand, using \cite[1.518.2]{gradshteyn},
\[
\log\cos\left(\frac{\pi u}{2K}\right)=-\sum_{k\ge1}\frac{2^{2k-1}\left(2^{2k}-1\right)\vert B_{2k}\vert}{k\left(2k\right)!}\left(\frac{\pi u}{2K}\right)^{2k}
\]
with, for $k\ge1,$
\[
\vert B_{2k}\vert=\left(-1\right)^{k-1}B_{2k}
\]
and
\[
\left(2^{2k}-1\right)B_{2k}=-kE_{2k-1}\left(0\right).
\]
We deduce 
\[
\log\cos\left(\frac{\pi u}{2K}\right)=-\frac{1}{2}\sum_{k\ge1}\frac{\left(-1\right)^{k}E_{2k-1}\left(0\right)}{\left(2k\right)!}\left(\frac{\pi u}{K}\right)^{2k}.
\]
Observe that
\begin{align*}
\sum_{n\ge1}q^{n}\frac{\sin^{2}\left(\frac{n\pi u}{2K}\right)}{n\left(1+\left(-q\right)^{n}\right)}&=\frac{1}{2}\sum_{n\ge1}q^{n}\frac{1-\cos\left(\frac{n\pi u}{K}\right)}{n\left(1+\left(-q\right)^{n}\right)}\\
&=\frac{1}{2}\sum_{n\ge1}q^{n}\frac{1}{n\left(1+\left(-q\right)^{n}\right)}\left(1-\sum_{p\ge0}\frac{\left(-1\right)^{p}}{\left(2p\right)!}\left(\frac{n\pi u}{K}\right)^{2p}\right)\\
&=-\frac{1}{2}\sum_{n\ge1}q^{n}\frac{1}{n\left(1+\left(-q\right)^{n}\right)}\left(\sum_{p\ge1}\frac{\left(-1\right)^{p}}{\left(2p\right)!}\left(\frac{n\pi u}{K}\right)^{2p}\right)\\
&=-\frac{1}{2}\sum_{p\ge1}\frac{\left(-1\right)^{p}}{\left(2p\right)!}\left(\frac{\pi u}{K}\right)^{2p}\sum_{n\ge1}\frac{n^{2p-1}q^{n}}{1+\left(-q\right)^{n}},
\end{align*}
so that we deduce the Taylor series expansion
\begin{align*}
\log \textnormal{cn}\left(u,k\right)&=-\frac{1}{2}\sum_{p\ge1}\frac{\left(-1\right)^{p}E_{2p-1}\left(0\right)}{\left(2p\right)!}\left(\frac{\pi u}{K}\right)^{2p}+2\sum_{p\ge1}\frac{\left(-1\right)^{p}}{\left(2p\right)!}\left(\frac{\pi u}{K}\right)^{2p}\sum_{n\ge1}\frac{n^{2p-1}q^{n}}{1+\left(-q\right)^{n}}\\
&=\sum_{p\ge1}\frac{\left(-1\right)^{p}}{\left(2p\right)!}\left(\frac{\pi u}{K}\right)^{2p}\left[-\frac{1}{2}E_{2p-1}\left(0\right)+2\sum_{n\ge1}\frac{n^{2p-1}q^{n}}{1+\left(-q\right)^{n}}\right].
\end{align*}
Finally, noticing that
\begin{align*}
\frac{1}{2}\int_{\mathbb{R}}\frac{\sin\left(u\sqrt{x}\right)}{\sqrt{x}}\frac{dx}{\cos\left(K\sqrt{x}\right)+\cosh\left(K'\sqrt{x}\right)}&=-\frac{d}{du}\log \textnormal{cn}\left(u,k\right)\\
&=\sum_{p\ge1}\frac{\left(-1\right)^{p-1}}{\left(2p-1\right)!}\left(\frac{\pi}{K}\right)^{2p}u^{2p-1}\left[-\frac{1}{2}E_{2p-1}\left(0\right)+2\sum_{n\ge1}\frac{n^{2p-1}q^{n}}{1+\left(-q\right)^{n}}\right].
\end{align*}
and expanding the sine term in Kuznetsov's integral
\[
\frac{1}{2}\int_{\mathbb{R}}\frac{\sin\left(u\sqrt{x}\right)}{\sqrt{x}}\frac{dx}{\cos\left(K\sqrt{x}\right)+\cosh\left(K'\sqrt{x}\right)}=
\frac{1}{2}\sum_{p\ge1}\frac{\left(-1\right)^{p-1}u^{2p-1}}{\left(2p-1\right)!}\int_{\mathbb{R}}\frac{x^{p-1}dx}{\cos\left(K\sqrt{x}\right)+\cosh\left(K'\sqrt{x}\right)}.
\]
This interchange of integration and summation is justified by the absolute and uniform convergence of the left integral on compact subsets of the domain $D:=\{u\in\mathbb{C}: \lvert\Re(u)\rvert<K\text{ and }\lvert\Im(u)\rvert<K'\}$. We deduce the Lambert series representation
\[
\frac{1}{2}\int_{\mathbb{R}}\frac{x^{p-1}dx}{\cos\left(K\sqrt{x}\right)+\cosh\left(K'\sqrt{x}\right)}=\left(\frac{\pi}{K}\right)^{2p}\left[-\frac{1}{2}E_{2p-1}\left(0\right)+2\sum_{n\ge1}\frac{n^{2p-1}q^{n}}{1+\left(-q\right)^{n}}\right].
\]
\end{proof}

\subsection{Eisenstein Series Representation}
Finally, we are able to produce an Eisenstein series representation for Kuznestov's integrals.
\begin{proposition}
The integral
\[
I_{p}^{+}=
\frac{1}{2}\int_{\mathbb{R}}\frac{x^p}{\cos( K\sqrt{x})+\cosh( K'\sqrt{x})}dx 
\]
has Eisenstein series expansion 
\[
I_{p}^{+}
= 
(-1)^p(2p+1)!
\left(\sideset{}{'}\sum_{m,n}\left(\frac{1}{\left(2mK+i\left(2n+1\right)K'\right)^{2p+2}}-\frac{1}{\left(\left(2m+1\right)K+i2nK'\right)^{2p+2}}\right)+\frac{1}{K^{2p+2}}-\frac{1}{\left(i K'\right)^{2p+2}}
\right),
\] with the notation   \[\sideset{}{'}\sum_{m,n}=\sum_{\begin{array}{c}
(m,n)\in \mathbb{Z}^2\\
(m,n)\ne(0,0)
\end{array}}.
\]
\end{proposition}
\begin{proof}
From \cite[vol.5 p.11]{tolke5}, with $c_{0,0}=0,\thinspace\thinspace c_{m,n}=1$
$\left(m,n\right)\ne\left(0,0\right)$
\begin{align*}
\frac{1}{2}\int_{\mathbb{R}}&\frac{\sin\left(u\sqrt{x}\right)}{\sqrt{x}}\frac{dx}{\cos\left(K\sqrt{x}\right)+\cosh\left(K'\sqrt{x}\right)}=\widetilde{\textnormal{nc}}\left(u,k\right)\\
&=-\eta_{1}K+i\eta_{2}K'
-\sum_{(m,n)\in \mathbb{Z}^{2}}\left[\frac{1}{2mK+i\left(2n+1\right)K'-u}-\frac{1}{\left(2m+1\right)K+i2nK'-u}+\frac{c_{m,n}\left(-K+i K'\right)}{\left(2mK+i2nK'\right)^{2}}\right]\\
&=-\eta_{1}K+i\eta_{2}K'-\left[\frac{1}{i K'-u}-\frac{1}{K-u}\right]\\
&\ \ \ \ \ -\sideset{}{'}\sum_{m,n}\left[\frac{1}{2mK+i\left(2n+1\right)K'-u}-\frac{1}{\left(2m+1\right)K+i2nK'-u}+\frac{-K+i K'}{\left(2mK+i2nK'\right)^{2}}\right]
\end{align*}

Expanding each term produces
\[
\frac{1}{2mK+i\left(2n+1\right)K'-u}-\frac{1}{\left(2m+1\right)K+i2nK'-u}=
\sum_{p\ge0}\frac{u^{p}}{\left(2mK+i\left(2n+1\right)K'\right)^{p+1}}-\frac{u^{p}}{\left(\left(2m+1\right)K+i2nK'\right)^{p+1}}
\]
and
\[
\frac{1}{i K'-u}-\frac{1}{K-u}=\sum_{p\ge0}\frac{u^{p}}{\left(i K'\right)^{p+1}}-\frac{u^{p}}{K^{p+1}},
\]
 from which we obtain
\[
\widetilde{\textnormal{nc}}\left(u,k\right)=\widetilde{\textnormal{nc}}\left(0,k\right)+\sum_{p\ge1}u^{p}\left(
\sideset{}{'}\sum_{m,n}\left(\frac{1}{\left(2mK+i\left(2n+1\right)K'\right)^{p+1}}-\frac{1}{\left(\left(2m+1\right)K+i2nK'\right)^{p+1}}\right)+\frac{1}{K^{p+1}}-\frac{1}{\left(i K'\right)^{p+1}}\right)
\]
with the constant term
\[
\widetilde{\textnormal{nc}}\left(0,k\right)=-\eta_{1}K+i\eta_{2}K'-\frac{1}{i K'}+\frac{1}{K}
-
\sideset{}{'}\sum_{m,n}
\left(\frac{1}{2mK+i(2n+1)K'}
-\frac{1}{(2m+1)K+i2nK'}+
\frac{-K+i K'}{\left(2mK+i2nK'\right)^{2}}\right)=0
\]  since $\widetilde{\textnormal{nc}}\left(u,k\right)$ is an odd function.
Identifying with 
\[
\widetilde{\textnormal{nc}}\left(u,k\right)=
\sum_{p\ge0}\frac{(-1)^p}{(2p+1)!}
I_{p}^{+}u^{2p+1}
\]
produces the Eisenstein series expansion for the integral
\[
I_{p}^{+}
= 
(-1)^p(2p+1)!
\left(\sideset{}{'}\sum_{m,n}\left(\frac{1}{\left(2mK+i\left(2n+1\right)K'\right)^{2p+2}}-\frac{1}{\left(\left(2m+1\right)K+i2nK'\right)^{2p+2}}\right)+\frac{1}{K^{2p+2}}-\frac{1}{\left(i K'\right)^{2p+2}}
\right).
\]
\end{proof}

\section{Extension of Kuznetsov's Result: The  Difference Case}\label{sec:jacobi-extension}
We will now endeavour to extend Kuznetsov's result to the case of the integrals \[I_{n}^{-}=\int_\mathbb{R}\frac{x^{n+1}}{\cos(K\sqrt{x})-\cosh(K'\sqrt{x})}dx\] where the addition in the denominator is replaced by a difference. 
\subsection{Generating Function}
We begin with a proof of the following theorem which gives a generating function for $I_n^{-}$.

\begin{theorem} Let $k\in(0,1)$. Then for $u\in\mathbb{C}$ satisfying $\lvert\Re(u)\rvert<K/2$ and $\lvert\Im(u)\rvert<K'/2$, we have
\begin{equation}\label{eq:kuznetsov-minus}
        \int_\mathbb{R}\frac{\sqrt{x}\sin(u\sqrt{x})}{\cos(K\sqrt{x})-\cosh(K'\sqrt{x})}dx=-8\frac{\textnormal{sn}^2(u,k)}{\textnormal{cd}^2(u,k)\textnormal{sd}(2u,k)}
        =-2\frac{d}{du}\widetilde{\textnormal{nc}}^{2}\left(u,k\right)
\end{equation}
with the log-derivative $\widetilde{\textnormal{nc}}(u,k) = \frac{d}{du} \log \textnormal{nc}(u,k)$.

\end{theorem}
\begin{proof}
    As in Kuznetsov's approach, we will establish \eqref{eq:kuznetsov-minus} for $u=v(K+iK')/2$ with $v\in(-1,1)$ and then extend by analytic continuation. Fix $v\in(-1,1)$ and let
    \begin{equation*}
        I:=\int_\mathbb{R}\frac{f(\sqrt{x}v(K+iK')/2)}{\sqrt{x}(\cos(\sqrt{x}K)-\cosh(\sqrt{x}K'))}dx.
    \end{equation*}
    We will begin with the change of variables $z=K\sqrt{x}/2$, which maps the contour $\mathbb{R}$ into the contour $L$ considered earlier. The result of the change of variables is therefore
    \begin{align*}
        I&=\frac{4}{K}\int_L\frac{f(vz(1+\tau))}{\cos(2z)-\cos(2z\tau)}dz\\
        &=-\frac{2}{K}\int_L\frac{f(vz(1+\tau))}{\sin(z(1+\tau))\sin(z(1-\tau))}dz
    \end{align*}
    where we have defined $\tau:=iK'/K$. Let $f(vz(1+\tau))=z^2g(vz(1+\tau))$ so that the singularity at $z=0$ is removable. We care only about the simple poles appearing in the first quadrant which have the form $z_n=\pi n/(1-\tau)$. Let us assume that $g$ is analytic in the first quadrant. By closing the contour with a quarter circle and showing that this arc term decays exponentially, it follows from Cauchy's residue theorem that
    \begin{align*}
        I&=-\frac{4\pi i}{K}\sum_{n=1}^\infty \text{Res}_{\pi n/(1-\tau)}\left(\frac{z^2g(vz(1+\tau))}{\sin(z(1+\tau))\sin(z(1-\tau))}\right)\\
        &=-\frac{4\pi^3 i}{K(1-\tau)^3}\sum_{n=1}^\infty(-1)^nn^2\frac{g(\pi nvt)}{\sin(\pi n t)},
    \end{align*}
    where we have defined $t:=\frac{1+\tau}{1-\tau}$. From \cite[(2.16)]{milne}, we have
    \begin{equation*}
        \text{sd}^2(vt\tilde{K},\tilde{k})=\frac{\tilde{E}-(\tilde{k}')^2\tilde{K}}{\tilde{k}^2(\tilde{k}')^2\tilde{K}}+\frac{2\pi^2}{\tilde{k}^2(\tilde{k}')^2\tilde{K}^2}\sum_{n\ge1}\frac{(-1)^nnq^n}{1-q^{2n}}\cos(\pi nvt),
    \end{equation*}
    where $q=e^{i\pi t}$ is the nome associated to the lattice parameter $t$, $\tilde{K}:=K(\tilde{k})$, and $\tilde{k}$ is defined by $t=iK(\tilde{k}')/K(\tilde{k})$. In other words, we let $t$ be the lattice parameter in \cite[(2.16)]{milne} and distinguish the associated modulus and elliptic integrals from those related to the lattice parameter $\tau$ by including a tilde. Differentiating with respect to $v$ produces
    \begin{align*}
        8\tilde{k}^2(\tilde{k}')^2\tilde{K}^3\text{sd}(vt\tilde{K},\tilde{k})\text{cd}(vt\tilde{K},\tilde{k})\text{nd}(vt\tilde{K},\tilde{k})=-4\pi^3i\sum_{n\ge1}(-1)^nn^2\frac{\sin(\pi nvt)}{\sin(\pi nt)}.
    \end{align*}
    By choosing $g(z)=\sin(z)$, it follows that
    \begin{equation*}
        I=\frac{8\tilde{k}^2(\tilde{k}')^2\tilde{K}^3}{K(1-\tau)^3}\text{sd}(vt\tilde{K},\tilde{k})\text{cd}(vt\tilde{K},\tilde{k})\text{nd}(vt\tilde{K},\tilde{k}).
    \end{equation*}
    With this choice of $g$, the integral becomes
    \begin{equation*}
        I=\frac{K^2}{4}\int_\mathbb{R}\frac{\sqrt{x}\sin(u\sqrt{x})}{\cos(K\sqrt{x})-\cosh(K'\sqrt{x})},
    \end{equation*}
    so that 
    \begin{equation*}
        J:=\int_\mathbb{R}\frac{\sqrt{x}\sin(u\sqrt{x})}{\cos(K\sqrt{x})-\cosh(K'\sqrt{x})}dx=\frac{32\tilde{k}^2(\tilde{k}')^2\tilde{K}^3}{K^3(1-\tau)^3}\text{sd}(vt\tilde{K},\tilde{k})\text{cd}(vt\tilde{K},\tilde{k})\text{nd}(vt\tilde{K},\tilde{k}).
    \end{equation*}
    The identities \cite[22.2.6-22.2.8]{nist} relate the Jacobi elliptic functions seen here to theta functions, and the identities \cite[22.2.2]{nist} relate the modulus, the complementary modulus, and the elliptic integral to theta functions. Making use of these identities, we produce
    \begin{equation*}
        J=\frac{4\pi^3}{K^3(1-\tau)^3}\theta_2^2(0,q)\theta_3^2(0,q)\theta_4^2(0,q)\frac{\theta_1(\pi vt/2,q)\theta_2(\pi vt/2,q)\theta_4(\pi vt/2,q)}{\theta_3^3(\pi vt/2,q)},
    \end{equation*}
    where the reader is reminded that the nome $q$ is with lattice parameter $t$. We will now transform the theta functions which are currently in terms of $t$ into theta functions which are in terms of $\tau$. To distinguish the two cases we write $\theta_i(w\vert t):=\theta_i(w,q)$ when the nome is in terms of lattice parameter $t$; that is, $q=e^{i\pi t}$. Now applying the theta function transformation identities \cite[20.7.26-20.7.29]{nist}, we have
    \begin{align*}
        J&=\frac{4\pi^3}{K^3(1-\tau)^3}\theta_2^2(0\vert t)\theta_3^2(0\vert t)\theta_4^2(0\vert t)\frac{\theta_1(\pi vt/2\vert t)\theta_2(\pi vt/2\vert t)\theta_4(\pi vt/2\vert t)}{\theta_3^3(\pi vt/2\vert t)}\\
        &=-\frac{4\pi^3}{K^3(1-\tau)^3}\theta_2^2(0\vert t+1)\theta_4^2(0\vert t+1)\theta_3^2(0\vert t+1)\frac{\theta_1(\pi vt/2\vert t+1)\theta_2(\pi vt/2\vert t+1)\theta_3(\pi vt/2\vert t+1)}{\theta_4^3(\pi vt/2\vert t+1)}\\
        &=-\frac{4\pi^3}{K^3(1-\tau)^3}\theta_2^2\left(0\bigg|\frac{2}{1-\tau}\right)\theta_4^2\left(0\bigg|\frac{2}{1-\tau}\right)\theta_3^2\left(0\bigg|\frac{2}{1-\tau}\right)\frac{\theta_1\left(\pi vt/2\bigg|\frac{2}{1-\tau}\right)\theta_2\left(\pi vt/2\bigg|\frac{2}{1-\tau}\right)\theta_3\left(\pi vt/2\bigg|\frac{2}{1-\tau}\right)}{\theta_4^3\left(\pi vt/2\bigg|\frac{2}{1-\tau}\right)}.
    \end{align*}
    Next we apply the product reduction formulae \cite[20.7.11-20.7.12]{nist} to obtain
    \begin{align*}
        J&=-\frac{\pi^3}{2K^3(1-\tau)^3}\theta_2^5\left(0\bigg| \frac{1}{1-\tau}\right)\theta_3^2\left(0\bigg| \frac{1}{1-\tau}\right)\theta_4^2\left(0\bigg| \frac{1}{1-\tau}\right)\frac{\theta_1\left(\pi vt/4\bigg| \frac{1}{1-\tau}\right)\theta_2\left(\pi vt/4\bigg| \frac{1}{1-\tau}\right)\theta_2\left(\pi vt/2\bigg| \frac{1}{1-\tau}\right)}{\theta_3^3\left(\pi vt/4\bigg| \frac{1}{1-\tau}\right)\theta_4^3\left(\pi vt/4\bigg| \frac{1}{1-\tau}\right)}.
    \end{align*}
    Applying the lattice parameter transformations \cite[20.7.30-20.7.33]{nist} produces
    \begin{align*}
        J&=\frac{\pi^3\theta_4^5\left(0\vert \tau-1\right)\theta_3^2\left(0\vert \tau-1\right)\theta_2^2\left(0\vert \tau-1\right)\theta_1\left(-\pi v(\tau+1)/4\vert \tau-1\right)\theta_4\left(-\pi v(\tau+1)/4\vert \tau-1\right)\theta_4\left(-\pi v(\tau+1)/2\vert \tau-1\right)}{2K^3\theta_3^3\left(-\pi v(\tau+1)/4\vert \tau-1\right)\theta_2^3\left(-\pi v(\tau+1)/4\vert \tau-1\right)}.
    \end{align*}
    Using the lattice parameter transformations \cite[20.7.26-20.7.29]{nist} once more, we have
    \begin{align*}
        J&=\frac{\pi^3\theta_3^5\left(0\vert \tau\right)\theta_4^2\left(0\vert \tau\right)\theta_2^2\left(0\vert \tau\right)\theta_1\left(-\pi v(\tau+1)/4\vert \tau\right)\theta_3\left(-\pi v(\tau+1)/4\vert \tau\right)\theta_3\left(-\pi v(\tau+1)/2\vert \tau\right)}{2K^3\theta_4^3\left(-\pi v(\tau+1)/4\vert \tau\right)\theta_2^3\left(-\pi v(\tau+1)/4\vert \tau\right)}.
    \end{align*}
    We can write the elliptic integral $K$ in terms of $\theta_3$ by applying \cite[22.2.2]{nist}, so that
    \begin{align*}
        J&=\frac{4\theta_4^2\left(0\vert \tau\right)\theta_2^2\left(0\vert \tau\right)\theta_1\left(-\pi v(\tau+1)/4\vert \tau\right)\theta_3\left(-\pi v(\tau+1)/4\vert \tau\right)\theta_3\left(-\pi v(\tau+1)/2\vert \tau\right)}{\theta_3\left(0\vert \tau\right)\theta_4^3\left(-\pi v(\tau+1)/4\vert \tau\right)\theta_2^3\left(-\pi v(\tau+1)/4\vert \tau\right)}.
    \end{align*}
    Now applying the duplication formula \cite[22.7.10]{nist} produces
    \begin{align*}
        J&=\frac{8\theta_2\left(0\vert \tau\right)\theta_4\left(0\vert \tau\right)\theta_1^2\left(-\pi v(\tau+1)/4\vert \tau\right)\theta_3^2\left(-\pi v(\tau+1)/4\vert \tau\right)\theta_3\left(-\pi v(\tau+1)/2\vert \tau\right)}{\theta_3^2\left(0\vert \tau\right)\theta_4^2\left(-\pi v(\tau+1)/4\vert \tau\right)\theta_2^2\left(-\pi v(\tau+1)/4\vert \tau\right)\theta_1\left(-\pi v(\tau+1)/2\vert \tau\right)},
    \end{align*}
    so that the definitions of the Jacobi elliptic functions \cite[22.2.4,22.2.7,22.2.8]{nist} finally give
    \begin{align*}
        J&=8\frac{\text{sn}^2(-Kv(\tau+1)/2,k)}{\text{cd}^2(-Kv(\tau+1)/2,k)\text{sd}(-Kv(\tau+1),k)}\\
        &=-8\frac{\text{sn}^2(Kv(\tau+1)/2,k)}{\text{cd}^2(Kv(\tau+1)/2,k)\text{sd}(Kv(\tau+1),k)},
    \end{align*}
    where we have used the facts that sd and sn are odd and cd is even. Recall that $v=2u/(K+iK')=2u/(K(\tau+1))$, so that we conclude
    \begin{equation*}
        \int_\mathbb{R}\frac{\sqrt{x}\sin(u\sqrt{x})}{\cos(K\sqrt{x})-\cosh(K'\sqrt{x})}dx=-8\frac{\text{sn}^2(u,k)}{\text{cd}^2(u,k)\text{sd}(2u,k)}
    \end{equation*}
    whenever $u=v(K+iK')/2$ with $v\in(-1,1)$. Moreover, the integral converges absolutely and uniformly on compact subsets of the domain $D:=\{u\in\mathbb{C}: \lvert\Re(u)\rvert<K/2\text{ and }\lvert\Im(u)\rvert<K'/2\}$ and therefore defines an analytic function on $D$. Meanwhile, by locating the poles of sn and the zeros of cd and sd using \cite[Table 22.4.1]{nist}, we see that this quotient of Jacobi elliptic functions is analytic on $D$. Then by the identity theorem, the identity \eqref{eq:kuznetsov-minus} holds on all of $D$.
    
    The alternate expression of the generating function in terms of the log-derivative of the $nc$ elliptic function is obtained by elementary transformations of the Jacobi elliptic functions as follows:
    observe that, by \cite[vol 3, (839)]{tolke3}
\begin{align*}
\frac{\textnormal{sn}\left(u,k\right)^{2}}{\textnormal{cd}\left(u,k\right)^{2}}\frac{1}{\textnormal{sd}\left(2u,k\right)}&=\frac{1-\textnormal{cn}\left(2u,k\right)}{1+\textnormal{dn}\left(2u,k\right)}\frac{1+\textnormal{dn}\left(2u,k\right)}{1+\textnormal{cn}\left(2u,k\right)}\frac{1}{\textnormal{sd}\left(2u,k\right)}\\
&=\frac{1-\textnormal{cn}\left(2u,k\right)}{1+\textnormal{cn}\left(2u,k\right)}\frac{1}{\textnormal{sd}\left(2u,k\right)}\\
&=\frac{1}{2}\frac{d}{du}\frac{1}{1+\textnormal{cn}\left(2u,k\right)}.
\end{align*}
Define ${\widetilde{{\textnormal{nc}}}(u,k)}=\frac{d}{du}\log \textnormal{nc}(u,k)$ so that, from \cite[vol 3, (848)]{tolke3}, \[\frac{1}{1+\textnormal{cn}\left(2u,k\right)}=\frac{1}{2}\left(1+\widetilde{\textnormal{nc}}^{2}\left(u,k\right)\right),\]
and we deduce
\[
\frac{\textnormal{sn}\left(u,k\right)^{2}}{\textnormal{cd}\left(u,k\right)^{2}}\frac{1}{\textnormal{sd}\left(2u,k\right)}=\frac{1}{4}\frac{d}{du}\widetilde{\textnormal{nc}}^{2}\left(u,k\right).
\]
\end{proof}

We can now obtain an evaluation of the Berndt-type integrals $I_{n}^{-}$ in terms of Jacobi elliptic functions.

\begin{corollary}
    Let $k\in(0,1)$ and $n\ge0$. Then
    \begin{equation}\label{eq:kuz-cor}
        \int_\mathbb{R}\frac{x^{n+1}}{\cos(K\sqrt{x})-\cosh(K'\sqrt{x})}dx=(-1)^{n+1}8\frac{d^{2n+1}}{du^{2n+1}}\frac{\textnormal{sn}^2(u,k)}{\textnormal{cd}^2(u,k)\textnormal{sd}(2u,k)}\bigg|_{u=0},
    \end{equation}
    and if $n=2m$ is even, it follows that
    \begin{equation*}
        \int_0^\infty\frac{x^{4m+3}}{\cos(x)-\cosh(x)}dx=-2\left(\frac{\pi^{\frac{3}{2}}}{2\Gamma^2(\frac{3}{4})}\right)^{4m+4}\frac{d^{4m+1}}{du^{4m+1}}\frac{\textnormal{sn}^2(u,1/\sqrt{2})}{\textnormal{cd}^2(u,1/\sqrt{2})\textnormal{sd}(2u,1/\sqrt{2})}\bigg|_{u=0}.
    \end{equation*}
\end{corollary}
\begin{proof}
    Note that the integrand in \eqref{eq:kuznetsov-minus} is in $L_1(0,\infty)$ as a function of $x$ for all $u$, as are its partial derivatives with respect to $u$. Then by differentiating \eqref{eq:kuznetsov-minus} under the integral sign with respect to $u$ using Leibniz integral rule \cite{folland1999real}, it follows that
    \begin{equation*}
        \sum_{m\ge0}(-1)^m\frac{u^{2m+1}}{(2m+1)!}\int_\mathbb{R}\frac{x^{m+1}}{\cos(K\sqrt{x})-\cosh(K'\sqrt{x})}dx=-8\frac{\textnormal{sn}^2(u,k)}{\textnormal{cd}^2(u,k)\textnormal{sd}(2u,k)},
    \end{equation*}
    so that differentiating $2n+1$ times recovers \eqref{eq:kuz-cor}. In the lemniscatic case $k=1/\sqrt{2}$, we have
    \begin{equation*}
        \int_\mathbb{R}\frac{x^{n+1}}{\cos(K\sqrt{x})-\cosh(K\sqrt{x})}dx=(-1)^{n+1}8\frac{d^{2n+1}}{du^{2n+1}}\frac{\textnormal{sn}^2(u,1/\sqrt{2})}{\textnormal{cd}^2(u,1/\sqrt{2})\textnormal{sd}(2u,1/\sqrt{2})}\bigg|_{u=0},
    \end{equation*}
    where $K=K(1/\sqrt{2})=\frac{\pi^{\frac{3}{2}}}{2\Gamma^2(\frac{3}{4})}$. Substituting $K\sqrt{x}\to x$, we have
    \begin{equation*}
        \frac{2}{K^{2n+4}}\int_\mathbb{L}\frac{x^{2n+3}}{\cos(x)-\cosh(x)}dx=(-1)^{n+1}8\frac{d^{2n+1}}{du^{2n+1}}\frac{\textnormal{sn}^2(u,1/\sqrt{2})}{\textnormal{cd}^2(u,1/\sqrt{2})\textnormal{sd}(2u,1/\sqrt{2})}\bigg|_{u=0},
    \end{equation*}
    and now substituting $ix\to x$ on the vertical part of the contour yields
    \begin{equation*}
        \frac{2}{K^{2n+4}}((-1)^n+1)\int_0^\infty\frac{x^{2n+3}}{\cos(x)-\cosh(x)}dx=(-1)^{n+1}8\frac{d^{2n+1}}{du^{2n+1}}\frac{\textnormal{sn}^2(u,1/\sqrt{2})}{\textnormal{cd}^2(u,1/\sqrt{2})\textnormal{sd}(2u,1/\sqrt{2})}\bigg|_{u=0},
    \end{equation*}
    so that when $n$ is even,
    \begin{align*}
        \int_0^\infty\frac{x^{2n+3}}{\cos(x)-\cosh(x)}dx&=-2K^{2n+4}\frac{d^{2n+1}}{du^{2n+1}}\frac{\textnormal{sn}^2(u,1/\sqrt{2})}{\textnormal{cd}^2(u,1/\sqrt{2})\textnormal{sd}(2u,1/\sqrt{2})}\bigg|_{u=0}\\
        &=-2\left(\frac{\pi^{\frac{3}{2}}}{2\Gamma^2(\frac{3}{4})}\right)^{2n+4}\frac{d^{2n+1}}{du^{2n+1}}\frac{\textnormal{sn}^2(u,1/\sqrt{2})}{\textnormal{cd}^2(u,1/\sqrt{2})\textnormal{sd}(2u,1/\sqrt{2})}\bigg|_{u=0}.
    \end{align*}
    Letting $n=2m$ completes the proof.
\end{proof}
This generalization of Kuznetsov’s result \eqref{eq:kuznetsov-minus} allow us to establish additional identities between the two families of integrals $I_{n}^{-}$ and
$I_{n}^{+}$.
\begin{corollary}
    From \eqref{eq:kuznetsov-minus} and \eqref{eq:kuznetsov+}, we deduce the binomial convolution identity
\[
I_{n}^{-}=-\sum_{p=0}^{n}\binom{2n+1}{2p+1}I_{p}^{+}I_{n-p}^{+}
\]
\end{corollary}
\begin{proof}
From the two generating functions
\[
J_{+}(u)=\int_{\mathbb{R}}\frac{\sin u\sqrt{x}}{\sqrt{x}}\frac{dx}{\cos K\sqrt{x}+\cosh K'\sqrt{x}}=2\widetilde{\textnormal{nc}}\left(u,k\right)
=\sum_{n\ge0}\frac{\left(-1\right)^{n}}{\left(2n+1\right)!}u^{2n+1}I_{n}^{+}\]

and
\[
J_{-}(u)=\int_{\mathbb{R}}\frac{\sqrt{x}\sin\left(u\sqrt{x}\right)dx}{\cos K\sqrt{x}+\cosh K'\sqrt{x}}=-2\frac{d}{du}\widetilde{\textnormal{nc}}^{2}\left(u,k\right)
=\sum_{n\ge0}\frac{\left(-1\right)^{n}}{\left(2n+1\right)!}u^{2n+1}I_{n}^{-},
\]

we deduce
\begin{align*}
J_{-}(u)&=-2\widetilde{\textnormal{nc}}\left(u,k\right)\frac{d}{du}2\widetilde{\textnormal{nc}}\left(u,k\right)\\
&=-\sum_{p,q\ge0}\frac{\left(-1\right)^{p}u^{2p+1}}{\left(2p+1\right)!}I_{p}^{+}\frac{\left(-1\right)^{q}u^{2q}}{\left(2q\right)!}I_{q}^{+}\\
&=-\sum_{p,n\ge0}\frac{\left(-1\right)^{n}u^{2n+1}}{\left(2p+1\right)!\left(2n-2p\right)!}I_{p}^{+}I_{n-p}^{+}
\end{align*}
so that
\[
I_{n}^{-}=-\sum_{p\ge0}\binom{2n+1}{2p+1}I_{p}^{+}I_{n-p}^{+}.
\]
\end{proof}

\subsection{The Polynomials of Lomont and Brillhart}
Evaluating the first values of the integrals $I_{n}^{-}$ produces
\[
I_{0}^{-}=-4,\,\,I_{1}^{-}=32(1-2k^2),\,\,I_{2}^{-}=-32(17-32k^2+32k^4),
\]
suggesting that, as in the case of the integrals $I_{n}^{+},$ they can be expressed as polynomials in the square of the elliptic modulus $k$. This is confirmed by the following result.
\begin{proposition}
Define the polynomials
\begin{equation}
\label{eq:Qnconv}
Q_{n}\left(x,y\right)=\sum_{p=0}^{n}\binom{2n+2}{2p+1}P_{p}\left(x,y\right)P_{n-p}\left(x,y\right)
\end{equation}
as the binomial convolution of the $P_{n}$ elliptic polynomials with
themselves. Then we have, for $n\ge0,$
\begin{equation}
\label{eq:In-Qn}
 I_{n}^{-}=\int_{\mathbb{R}}\frac{x^{n+1}}{\cos\left(K\sqrt{x}\right)-\cosh\left(K'\sqrt{x}\right)}dx=\left(-2\right)^{n+1}Q_{n}\left(1-2k^{2},4k^{4}-4k^{2}+4\right).
\end{equation}
 
\end{proposition}

\begin{proof}
Start from
\begin{equation}
\int_{\mathbb{R}}\frac{\sqrt{x}\sin\left(u\sqrt{x}\right)}{\cos\left(K\sqrt{x}\right)-\cosh\left(K'\sqrt{x}\right)}dx=-2\frac{d}{du}\widetilde{\textnormal{nc}}^{2}\left(u,k\right).\label{eq:gf}
\end{equation}
With $x=1-2k^{2}$ and $y=4k^{4}-4k^{2}+4,$ we have
\[
\widetilde{\textnormal{nc}}\left(u,k\right)=\sum_{n\ge0}2^{n}P_{n}\left(x,y\right)\frac{u^{2n+1}}{\left(2n+1\right)!},
\]
so that
\begin{align*}
\widetilde{\textnormal{nc}}^{2}\left(u,k\right)&=\sum_{p,q\ge0}2^{p+q}P_{p}\left(x,y\right)P_{q}\left(x,y\right)\frac{u^{2p+1}}{\left(2p+1\right)!}\frac{u^{2q+1}}{\left(2q+1\right)!}\\
&=\sum_{n,p\ge0}2^{n}P_{p}\left(x,y\right)P_{n-p}\left(x,y\right)\frac{u^{2n+2}}{\left(2p+1\right)!\left(2n-2p+1\right)!}\\
&=\sum_{n\ge0}2^{n}Q_{n}\left(x,y\right)\frac{u^{2n+2}}{\left(2n+2\right)!},
\end{align*}
where we have defined
\[
Q_{n}\left(x,y\right)=\sum_{p=0}^{n}\binom{2n+2}{2p+1}P_{p}\left(x,y\right)P_{n-p}\left(x,y\right).
\]
We deduce, still with $x=1-2k^{2}$ and $y=4k^{4}-4k^{2}+4,$
\begin{align*}
\int_{\mathbb{R}}\frac{\sqrt{x}\sin\left(u\sqrt{x}\right)}{\cos\left(K\sqrt{x}\right)-\cosh\left(K'\sqrt{x}\right)}dx&=-2\frac{d}{du}\widetilde{\textnormal{nc}}^{2}\left(u,k\right)\\
&=-\sum_{n\ge0}2^{n+1}Q_{n}\left(x,y\right)\frac{u^{2n+1}}{\left(2n+1\right)!}.
\end{align*}
The Taylor expansion
\[
\int_{\mathbb{R}}\frac{\sqrt{x}\sin\left(u\sqrt{x}\right)}{\cos\left(K\sqrt{x}\right)-\cosh\left(K'\sqrt{x}\right)}dx
=\sum_{n\ge0}\frac{\left(-1\right)^{n}u^{2n+1}}{\left(2n+1\right)!}\int_{\mathbb{R}}\frac{x^{n+1}}{\cos\left(K\sqrt{x}\right)-\cosh\left(K'\sqrt{x}\right)}dx
\]
produces (recall that the integral converges absolutely and uniformly on compact subsets of the domain $D:=\{u\in\mathbb{C}: \lvert\Re(u)\rvert<K/2\text{ and }\lvert\Im(u)\rvert<K'/2\}$, justifying the swapping of the summation and integration), upon identification with \eqref{eq:gf},
\[
\int_{\mathbb{R}}\frac{x^{n+1}}{\cos\left(K\sqrt{x}\right)-\cosh\left(K'\sqrt{x}\right)}dx=\left(-2\right)^{n+1}Q_{n}\left(x,y\right)
\]
which is the desired result.
\end{proof}
The first values  of the polynomials  $Q_n(x,y)$ are 
\[
Q_0(x,y)=2,\,\,Q_1(x,y)=8x,\,\,Q_2(x,y)=20x^2+12y.
\]
Notice that these polynomials appear in  \cite[Table 4.2]{Lomont} where they differ from the ones defined here by a factor 2.
\begin{corollary}
If $k^{2}$ is a rational number, then the integral $I_{n}^{-}=\int_{\mathbb{R}}\frac{x^{n+1}}{\cos\left(K\sqrt{x}\right)-\cosh\left(K'\sqrt{x}\right)}dx$ is a rational number. Moreover, since 
\[
I_{1}^{-}=32\left(1-2k^{2}\right),\thinspace\thinspace I_{2}^{-}=-32\left(17-32k^{2}+32k^{4}\right),
\]
the value of $I_n^{-}$ for $n\ge 3$ is a polynomial function of the initial values $I_1^{-}$ and $I_2^{-}$ given by
\[
I_{n}^{-}=\left(-2\right)^{n+1}Q_{n}\left(\frac{1}{32}I_{1}^{-},\frac{1}{8}\left(15-\frac{I_{2}^{-}}{32}\right)\right).
\]    
Finally, with $x=1-2k^{2}$ and $y=4k^{4}-4k^{2}+4,$ the integrals $I_{n}^{-}$  satisfy the recurrence
\begin{equation}
I_{n+2}^{-}=-8(1-2k^2)I_{n+1}^{-}-3\sum_{j=0}^{n}\binom{2n+4}{2j+2}I_{j}^{-}I_{n-j}^{-} 
\label{eq:recurrenceIn-}
\end{equation}
with initial values $I_{0}^{-}=-4$ and $I_{1}^{-}=32(1-2k^2).$
\begin{proof}
    The integrality of the coefficients of $Q_{n}(x,y)$  is a consequence of the integrality of the coefficients of $P_{n}(x,y)$  and of their expression \eqref{eq:Qnconv}. The expression of $I_{n}^{-}$  in terms of the initial values $I_{1}^{-}$ and $I_{2}^{-}$  is obtained from \eqref{eq:In-Qn} by substituting $x$  and $y$  as functions of $I_{1}^{-}$ and $I_{2}^{-}$  respectively. Recurrence \eqref{eq:recurrenceIn-} is deduced from recurrence \cite[(4.33)]{Lomont}  where the $Q_n$ that appear there are half those defined here. \end{proof}
\end{corollary}
Additional recursive identities between the two sets of polynomials $P_{n}(x,y)$  and $Q_{n}(x,y)$  can be found in \cite[Chapter 4]{Lomont}; they induce identities between the two sets of integrals $I_{n}^{+}$ and $I_{n}^{-}.$  For example, \cite[(4.42)]{Lomont}  produces the identity, for $n\ge 0,$
\[
\sum_{j=0}^{n}(n-3j)\binom{2n+3}{2j+1}I_{j}^{+}I_{n-j}^{-}=0.
\]
\subsection{A Probabilistic Approach}
The previous results can be given a probabilistic interpretation. Define the discrete random variable $X$ by
\[
\Pr\left\{ X=\pm \frac{\left(2n-1\right)\pi}{c}\right\} =\frac{\pi}{k'K'}\frac{q^{n-\frac{1}{2}}}{1+q^{2n-1}},
\]
where $n\ge1$, $c(k)=\frac{K\left(k\right)}{K\left(k'\right)},$ and $q=e^{-\pi\frac{K\left(k\right)}{K'\left(k\right)}}$. Then we can compute the moment generating function and the cumulants of $X$.

\begin{proposition}
The moment generating function of $X$ is
\[
\varphi_{X}\left(u\right)=\mathbb{E}e^{uX}=\textnormal{nc}\left(u,k\right)
\]
\end{proposition}
\begin{proof}
The proof is a straightfoward computation. We have
\begin{align*}
\mathbb{E}e^{uX}&=\sum_{n\in\mathbb{Z}}p_{n}e^{ux_{n}}=\sum_{n\ge1}\frac{\pi}{k'K'}\frac{q^{n-\frac{1}{2}}}{1+q^{2n-1}}e^{\frac{\left(2n-1\right)\pi}{c}u}+\sum_{n\ge1}\frac{\pi}{k'K'}\frac{q^{n-\frac{1}{2}}}{1+q^{2n-1}}e^{-\frac{\left(2n-1\right)\pi}{c}u}\\
&=\frac{2\pi}{k'K'}\sum_{n\ge1}\frac{q^{n-\frac{1}{2}}}{1+q^{2n-1}}\cosh \left( \frac{\left(2n-1\right)\pi}{c}u\right)=\textnormal{nc}\left(u,k\right)
\end{align*}
by \cite[p.19]{tolke2}.
\end{proof}
\begin{proposition}
For $n\ge1$, the cumulants $\kappa_{2n}$ of the discrete random variable $X$ coincide with the moments $I_{n-1}^{+}$ of a continuous random variable $Y$ with the probability density
\[
f_{Y}\left(z\right)=\frac{1}{2}\frac{1}{\cos(K\sqrt{z})+\cosh(K'\sqrt{z})}.
\]
\end{proposition}
\begin{proof}
The positivity of the function $f_{Y}$ and the fact that $\int_{\mathbb{R}}f_{Y}(y)dy=1$ make $f_{Y}$ a probability density function. Moreover, the cumulants $\kappa_{X}$ of $X$ are defined by
\begin{align*}
\log (\textnormal{nc}\left(i u,k\right))&=\sum_{n\ge1}\kappa_{n}\frac{(i u)^{n}}{n!}=\sum_{n\ge1}\kappa_{2n}\left(-1\right)^{n}\frac{u^{2n}}{\left(2n\right)!}
\end{align*}
since $\log(\textnormal{nc}\left(u,k\right))$ is an even function of $u$. Moreover,
by \cite[p.19]{tolke2},
\[
\frac{d}{du}\log (\textnormal{nc}\left(u,k\right))=\sum_{n\ge0}\frac{u^{2n+1}}{\left(2n+1\right)!}I_{n}^{+}
\]
with 
\[
I_{n}^{+}=\frac{1}{2}\int\frac{x^{n}}{\cos(K\sqrt{x})+\cosh(K'\sqrt{x})}dx
\]
so that, since $\textnormal{nc}\left(0,k\right)=1,$ we have
\[
\log(\textnormal{nc}\left(i u,k\right))=\sum_{n\ge0}\left(-1\right)^{n}\frac{u^{2n+2}}{\left(2n+2\right)!}I_{n}^{+},
\]
and it follows that $\kappa_{2n}=I_{n-1}^{+}$ for $n\ge1$.
\end{proof}

\subsection{Lambert Series Representation}
As in Kuznetsov's case, our extension produces a Lambert series representation for Berndt-type integrals. 
\begin{proposition}
For $n\ge0$, the $I_{n}^{-}$ integrals have the following Lambert series representation
\[
\int_{\mathbb{R}}\frac{x^{n+1}}{\cos\left(K\sqrt{x}\right)-\cosh\left(K'\sqrt{x}\right)}dx=-\left(\frac{\pi}{K}\right)^{2n+4}\left(E_{2n+3}\left(0\right)+4\sum_{m\ge1}\frac{m^{2n+3}q^{m}}{1-\left(-q\right)^{m}}\right)
\]
with $q=e^{-\pi \frac{K'(k)}{K(k)}}$ and $E_{n}(x)$ the Euler polynomial of degree $n$ defined by the generating function
\[\sum_{n\ge 0}\frac{E_{n}(x)}{n!}z^n = \frac{2}{e^{z}+1}e^{zx}.\]
\end{proposition}
\begin{proof}
From \cite[p. 134]{tolke2}, the Lambert series expansion for the generating function is
\[
\frac{d}{du}\bar{\textnormal{nc}}^{2}\left(u,k\right)=\frac{\pi^{3}}{4K^{3}}\frac{\tan\left(\frac{\pi u}{2K}\right)}{\cos^{2}\left(\frac{\pi u}{2K}\right)}+\frac{2\pi^{3}}{K^{3}}\sum_{n\ge1}\frac{n^{2}q^{n}\sin\left(\frac{n\pi u}{K}\right)}{1-\left(-q\right)^{n}}.
\]
The Taylor expansion 
\[
\tan z=\sum_{n\ge1}\left(-1\right)^{n-1}\frac{2^{2n}\left(2^{2n}-1\right)B_{2n}}{2n!}z^{2n-1}
\]
together with the expression of the scaled Bernoulli numbers as Euler polynomials
\[
\left(2^{2n}-1\right)B_{2n}=-nE_{2n-1}\left(0\right)
\]
produces
\[
\tan z=\sum_{n\ge1}\left(-1\right)^{n}\frac{2^{2n-1}E_{2n-1}\left(0\right)}{2n-1!}z^{2n-1}.
\]
Since
\[
\frac{\tan z}{\cos^{2}z}=\frac{1}{2}\frac{d^{2}}{dz^{2}}\tan z,
\]
we deduce
\[
\frac{\tan z}{\cos^{2}z}=\frac{1}{2}\frac{d^{2}}{dz^{2}}\sum_{n\ge1}\left(-1\right)^{n}\frac{2^{2n-1}E_{2n-1}\left(0\right)}{(2n-1)!}z^{2n-1}
=\sum_{n\ge1}\left(-1\right)^{n-1}\frac{2^{2n}E_{2n+1}\left(0\right)}{\left(2n-1\right)!}z^{2n-1}.
\]
Moreover, expanding
\[
\sum_{n\ge1}\frac{n^{2}q^{n}\sin\left(\frac{n\pi u}{K}\right)}{1-\left(-q\right)^{n}}=\sum_{n\ge1}\frac{n^{2}q^{n}}{1-\left(-q\right)^{n}}\sum_{p\ge0}\frac{\left(-1\right)^{p}}{\left(2p+1\right)!}\left(\frac{n\pi u}{K}\right)^{2p+1}
=\sum_{p\ge0}\frac{\left(-1\right)^{p}}{\left(2p+1\right)!}\left(\frac{\pi u}{K}\right)^{2p+1}\sum_{n\ge1}\frac{n^{2p+3}q^{n}}{1-\left(-q\right)^{n}}
\]
produces
\begin{align*}
2\frac{\textnormal{sn}\left(u,k\right)^{2}}{\textnormal{cd}\left(u,k\right)^{2}}\frac{1}{\textnormal{sd}\left(2u,k\right)}&=\frac{\pi^{3}}{8K^{3}}\sum_{p\ge0}\left(-1\right)^{p}\frac{2^{2p+2}E_{2p+3}\left(0\right)}{\left(2p+1\right)!}\left(\frac{\pi u}{2K}\right)^{2p+1}
+\frac{\pi^{3}}{K^{3}}\sum_{p\ge0}\frac{\left(-1\right)^{p}}{\left(2p+1\right)!}\left(\frac{\pi u}{K}\right)^{2p+1}\sum_{n\ge1}\frac{n^{2p+3}q^{n}}{1-\left(-q\right)^{n}}\\
&=\sum_{p\ge0}\left(-1\right)^{p}\left(\frac{\pi}{K}\right)^{2p+4}\frac{u^{2p+1}}{\left(2p+1\right)!}\left[\frac{1}{4}E_{2p+3}\left(0\right)+\sum_{n\ge1}\frac{n^{2p+3}q^{n}}{1-\left(-q\right)^{n}}\right].
\end{align*}
It follows that
\[
2\frac{d^{2n+1}}{du^{2n+1}}\frac{\textnormal{sn}\left(u,k\right)^{2}}{\textnormal{cd}\left(u,k\right)^{2}}\frac{1}{\textnormal{sd}\left(2u,k\right)}_{u=0}=\left(-1\right)^{n}\left(\frac{\pi} {K}\right)^{2n+4}\left[\frac{1}{4}E_{2n+3}\left(0\right)+\sum_{m\ge1}\frac{m^{2n+3}q^{m}}{1-\left(-q\right)^{m}}\right]
\]
and
\[
\int_{\mathbb{R}}\frac{x^{n+1}}{\cos\left(K\sqrt{x}\right)-\cosh\left(K'\sqrt{x}\right)}dx=-\left(\frac{\pi}{K}\right)^{2n+4}\left[E_{2n+3}\left(0\right)+4\sum_{m\ge1}\frac{m^{2n+3}q^{m}}{1-\left(-q\right)^{m}}\right].
\]
\end{proof}    
\subsection{Eisenstein Series Representation}
\begin{proposition}
For $n\ge0,$ the integrals $I_{n}^{-}$ have the Eisenstein series expansion 
\[
I_{n}^{-}=\int_{\mathbb{R}}\frac{x^{n+1}}{\cos(K\sqrt{x})-\cosh(K'\sqrt{x})}dx=\left(-1\right)^{n+1}2\left(2n+3\right)!\sum_{(p,q)\in\mathbb{Z}^2}\frac{1}{\left(\left(2q+p-1\right)K+i pK'\right)^{2n+4}}.
\]
\end{proposition}
\begin{proof}
From \cite[(771)]{tolke3}, 
\[
\frac{d}{du}\widetilde{\textnormal{nc}}^{2}\left(u,k\right) = \frac{d}{du}\wp_{6}\left(u,k\right)
\]
where the Weierstra{\ss} $\wp_{6}$ function is the specialization  
\[
\wp_{6}\left(u,k\right) = \wp(u+K,K,\frac{K+\imath K'}{2})
\]
of the Weierstra{\ss} function $\wp$
with the double series representation \cite[p.10]{tolke5}
\[
\wp(u+K,K,\frac{K+\imath K'}{2})=\sum_{(p,q)\in\mathbb{Z}^2}\frac{1}{\left(2qK+ 2p \frac{K+\imath K'}{2}-z-K\right)^{2}}-\frac{c_{p,q}}{\left(2qK+ 2p \frac{K+\imath K'}{2}\right)^{2}}
\]
where
\[
c_{p,q}=\begin{cases}
0, & p=q=0\\
1, & \text{else}
\end{cases}
.\]
We deduce

\[
\frac{d}{du}\widetilde{\textnormal{nc}}^{2}\left(u,k\right) = \frac{d}{du}\wp(u+K,K,\frac{K+\imath K'}{2})=
2\sum_{(p,q)\in\mathbb{Z}^2}\frac{1}{\left(\left(2q+p-1\right)K+i pK'-u\right)^{3}}.
\]
and the Taylor expansion in $u$ 
\[
\frac{d}{du}\widetilde{\textnormal{nc}}^{2}\left(u,k\right)=2\sum_{n\ge0}\binom{n+2}{2}u^{n}\sum_{(p,q)\in\mathbb{Z}^2}\frac{1}{\left(\left(2q+p-1\right)K+i pK'\right)^{n+3}}.
\]
Comparing with the expansion
\[
\int_{\mathbb{R}}\frac{\sqrt{x}\sin(u\sqrt{x})}{\cos(\sqrt{x})-\cosh(K'\sqrt{x})}dx=\sum_{n\ge0}\frac{\left(-1\right)^{n}}{\left(2n+1\right)!}u^{2n+1}I_{n}^{-}=-2\frac{d}{du}\widetilde{\textnormal{nc}}^{2}\left(u,k\right)
\]
where
\[
I_{n}^{-}=\int_{\mathbb{R}}\frac{x^{n+1}}{\cos(K\sqrt{x})-\cosh(K'\sqrt{x})}dx,
\]
we obtain
\[
\frac{\left(-1\right)^{n}}{\left(2n+1\right)!}I_{n}^{-}=-4\binom{2n+3}{2}\sum_{(p,q)\in\mathbb{Z}^2}\frac{1}{\left(\left(2q+p-1\right)K+i pK'\right)^{2n+4}}
\]
or
\[
I_{n}^{-}=\int_{\mathbb{R}}\frac{x^{n+1}}{\cos(K\sqrt{x})-\cosh(K'\sqrt{x})}dx=\left(-1\right)^{n+1}2\left(2n+3\right)!\sum_{(p,q)\in\mathbb{Z}^2}\frac{1}{\left(\left(2q+p-1\right)K+i pK'\right)^{2n+4}}.
\]
\end{proof}

\section{Arithmetical results}\label{section:Arithmetic}
\subsection{Modulo 10 results}
Lomont and Brillhart \cite{Lomont} produce some arithmetical results about the elliptic
polynomials, such as the modular identities \cite[(4.54),(4.56)]{Lomont}
\begin{equation}
P_{2n}\left(x,y\right)\equiv y^{n}\mod{10}\label{eq:P2nmod}
\end{equation}
and
\begin{equation}
Q_{2n}\left(x,y\right)\equiv12y^{n}\mod{10}
\label{eq:Q2nmod}    
\end{equation}
that can be lifted to Kuznetsov's integrals in the lemniscatic case.
\begin{proposition}
In the lemniscatic case, Kuznetsov's integrals satisfy
\[
I_{2n}^{+}\equiv2^{n}\mod{10}
\]
and
\[
I_{2n}^{-}\equiv6\times2^{n}\mod{10}
\]    
\end{proposition}
\begin{proof}
In the lemniscatic case, $x=0$ and $y=3$ so that (\ref{eq:P2nmod})
produces
\[
P_{2n}\left(x,y\right)\equiv3^{n}\mod{10}
\]
so that 
\[
I_{2n}^{+}=\left(-2\right)^{2n}P_{2n}\equiv12^{n}\mod{10}\equiv2^{n}\mod{10}.
\]
Moreover, (\ref{eq:Q2nmod}) produces
\[
Q_{2n}\left(x,y\right)\equiv12\times3^{n}\mod{10}
\]
so that
\[
I_{2n}^{-}=\left(-2\right)^{2n+1}Q_{2n}\equiv-4\times12^{n}\mod{10\equiv2^{n+4}\mod{10}}
\]
\[
\equiv6\times2^{n}\mod{10}
\]
\end{proof}
This result extends to more general values of the elliptic modulus as follows.
\begin{proposition}
Assume that the elliptic modulus $k$ is such that
\[
4\left(k^{4}-k^{2}+1\right)=\frac{p}{q},
\]
a rational number with $p\in\mathbb{Z},q\in\mathbb{Z}$ and $gcd\left(p,q\right)=1.$ Then
\[
q^{n}P_{2n}\left(x,y\right)\equiv p^{n}\mod{10}
\]
and as a consequence
\[
q^{n}I_{2n}^{+}\equiv\left(4p\right)^{n}\mod{10}
\]    
\end{proposition}
\begin{proof}
The polynomial $P_{2n}(x,y)$ is expressed as
\[
P_{2n}(x,y)=\sum a_k y^{k}(y-3)^{n-k}
\] 
with $a_k \equiv \begin{cases}
    0 \mod{10}, k\ne0\\
    1 \mod{10}, k=0
\end{cases}$
With $y=\frac{p}{q},$ we deduce
\[
q^n P_{2n}(x,y)=\sum a_k p^{k}(p-3q)^{n-k}
\] 
so that
\[
q^n P_{2n} \equiv p^n \mod{10}.
\]
\end{proof}
For example, let us take
$
p=13,q=4
$
so that the elliptic modulus is
$
k=\frac{1}{2}
$
and we deduce
\[
I_{2n}^{+}\equiv3^{n}\mod{10}.
\]
This is confirmed numerically as, with $k=1/2,$
\[
I_0^{+}=1\,\,I_2^{+}=13,\,\,I_4^{+}=4249,\dots
\]
\subsection{Modulo 3 results}
Lomont and Brillhart also propose base 3 modularity results such as \cite[(4.59)]{Lomont}
\[
P_{3n+1}(x,y)\equiv x^{3n+1} \mod{3}
\]
from which we deduce
\begin{proposition}
    Assume that the elliptic modulus $k$ is such that
\[
2k^{2}-1=\frac{p}{q},
\]
a rational number with $p\in\mathbb{Z},q\in\mathbb{Z}$ and $\textnormal{gcd}\left(p,q\right)=1,$  then Kuznetsov's integrals satisfy
    \[
    q^{3n+1}I_{3n+1}^{+}\equiv (-2p)^{3n+1} \mod{3}
    \]
\end{proposition}
\begin{proof}
The polynomial $P_{3n+1}(x,y)$ reads
\[
P_{3n+1}(x,y)=a_{3n+1}x^{3n+1}+\sum_{k=1}^{n} b_k x^{3n+1-2k}y^k
\] 
with $a_{3n+1} \equiv 1 \mod{3}$ and $b_{k} \equiv 0 \mod{3}$ so that, with $x=2k^{2}-1=\frac{p}{q}$ and $y=3+x^2,$
\[
q^{3n+1}P_{3n+1}(x,y)=a_{3n+1}p^{3n+1}+\sum_{k=1}^{n} b_k p^{3n+1-2k}q^{2k}(3+\frac{p^2}{q^2})^k
\]
\[
=a_{3n+1}p^{3n+1}+\sum_{k=1}^{n} b_k p^{3n+1-2k}(3q^2+p^2)^k
\]
We deduce
\[
q^{3n+1}P_{3n+1}(x,y)\equiv p^{3n+1} \mod{3}
\]
Since
\[
I_{3n+1}^{+}=(-2)^(3n+1)P_{3n+1},
\]we deduce the result.
\end{proof}
For example, $p=1,\,\,q=2$  produce $k=\frac{1}{2}$  and we deduce
\[
    I_{3n+1}^{+}\equiv (-1)^{n+1} \mod{3}
    \]
Numerically, $I_1^{+}=-1,\,\,I_4^{+}=4249,\,\,I_7^{+}=-602994637,\dots$
Other identities modulo 3 and 10 are available in Lomont and Brillhart.
\section{Conclusion and Perspectives}\label{sec:conclusion}
Berndt-type integrals are surprisingly rich with connections to various special functions. Central to our investigation was the utilization of the Barnes zeta function, which provided a powerful framework for evaluating Berndt-type integrals in terms of a multiple series representation. This approach not only extended the scope of known evaluations but also allowed us to exhibit the analytic continuation of Berndt-type integrals. Moreover, we have extended Kuznetsov's direct evaluation of the integral considered by Ismail and Valent related to the Nevanlinna parametrization of solutions to a certain indeterminate moment problem \cite{berg1994}. Kuznetsov's evaluations provide a direct link between the generating functions of Jacobi elliptic functions and integrals involving hyperbolic and trigonometric functions. By leveraging Kuznetsov's findings, we have demonstrated specific instances where Berndt-type integrals can be expressed in terms of Jacobi elliptic functions, thereby establishing a richer analytical understanding of their nature. This not only enhances our ability to compute these integrals but also opens avenues for exploring their broader implications within the realm of special functions and mathematical physics.

Through our investigations, we have demonstrated the versatility of a variety of approaches in handling Berndt-type integrals, from Lambert series representations to Barnes zeta function evaluations. Each method offers unique insights and avenues for further exploration, revealing connections to broader classes of mathematical objects such as moment polynomials and Bernoulli-Barnes polynomials. In essence, the study of Berndt-type integrals exemplifies the enduring quest within mathematics to unify seemingly disparate concepts and reveal underlying symmetries. As we continue to delve deeper into their properties and connections—whether through the lens of Barnes zeta functions, Jacobi elliptic functions, or other mathematical frameworks—we anticipate further revelations and applications across disciplines. Thus, Berndt-type integrals not only present challenges in integration theory but also serve as gateways to new mathematical connections.

The fact that Lambert series can be expressed as polynomials in the elliptic modulus appears in Chapter 17 of Ramanujan's work \cite[Chapter 17]{berndt1991} and more recently in S. Cooper's work \cite{cooper2009}. As we have shown, the study of these elliptic polynomials by Lomont and Brillhart gives detailed information  about the Lambert series and, in turn, about Berndt's integrals. A systematic way to link Lambert series to Eisenstein series is the object of several articles by Ling \cite{ling1974,ling1975}. In a future companion paper, our ultimate goal is to produce an integral representation for all elementary Lambert series $\sum_{n\ge 1} n^s \frac{(\pm 1)^n q^{a_n}}{1\pm(\pm q)^{b_n}}$ with $(a_n,b_n) \in  \left\{n,2n\right\}$, their associated Eisenstein series and Barnes zeta series. A first example is
\[
4\left(\frac{\pi}{K(k)}\right)^{2n+4}\sum_{m\ge1}\frac{m^{2n+3}q^m}{1-q^{2m}} =-\int_{\mathbb{R}}\frac{x^{n+1}\cosh (K'\sqrt{x})}{\cos^2(K\sqrt{x})-\cosh^2 (K'\sqrt{x})}dx.
\]
As in the examples shown in this article, the integral representation contains symmetries that produce non elementary identities for these Lambert series.

\bibliographystyle{plainurl}

\bibliography{ref}


\section*{Appendix}\label{sec:euler-bernoulli}

An alternative to the approach in Section~\ref{sec:barnes} is to make use of a symbolic technique involving the Bernoulli-Barnes and Euler-Barnes polynomials \cite{bayad2013,jiu2016}. With symbolic notations, the Bernoulli-Barnes polynomials are 
\[
B_{p}^{\left(2\right)}\left(x;a_{1},a_{2}\right)=\left(x+a_{1}B_{1}+a_{2}B_{2}\right)^{p}
\]
and the Euler-Barnes polynomials are
\[
E_{p}^{\left(2\right)}\left(x;a_{1},a_{2}\right)=\left(x+a_{1}E_{1}+a_{2}E_{2}\right)^{p},
\]
where $B_i$ is the $i$-th Bernoulli number and $E_i$ is the $i$-th Euler number. Their generating functions are
\[
\sum_{n\ge0}\frac{B_{n}^{\left(2\right)}\left(x;a_{1},a_{2}\right)}{n!}z^{n}=e^{zx}\frac{a_{1}z}{e^{a_{1}z}-1}\frac{a_{2}z}{e^{a_{2}z}-1}
\]
and
\[
\sum_{n\ge0}\frac{E_{n}^{\left(2\right)}\left(x;a_{1},a_{2}\right)}{n!}z^{n}=e^{zx}\frac{2}{e^{a_{1}z}+1}\frac{2}{e^{a_{2}z}+1}.
\]
We will make use of the following lemma.
\begin{lemma}\label{lemma:expansion-lem}
The following expansions hold:
\[
\frac{x^{2}}{\cosh x-\cos x}=\frac{i x^{2}}{2}\csc\left(\sqrt{\frac{i}{2}}x\right)\textnormal{csch}\left(\sqrt{\frac{i}{2}}x\right)
=\sum_{p\ge0}\frac{x^{p}}{p!}g_{1}\left(p\right)
\]
with
\[
g_{1}\left(p\right)=\left(\sqrt{2i}\right)^{p}\left(\left(B_{1}+\frac{1}{2}\right)+i\left(B_{2}+\frac{1}{2}\right)\right)^{p}=\left(\sqrt{2i}\right)^{p}B_{p}^{\left(2\right)}\left(\frac{1+i}{2};1,i\right)
\]
and
\[
\frac{1}{\cosh x+\cos x}=\frac{1}{2}\textnormal{sech}\left(\sqrt{\frac{i}{2}}x\right)\sec\left(\sqrt{\frac{i}{2}}x\right)
=\sum_{p\ge0}\left(-1\right)^{p}\frac{x^{2p}}{2p!}g_{2}\left(2p\right)
\]
with
\[
g_{2}\left(2p\right)=\frac{1}{2}\left(\sqrt{\frac{i}{2}}\right)^{2p}E_{2p}^{\left(2\right)}\left(0;1,i\right).
\]
\end{lemma}

Ramanujan's master theorem \cite{RMT} (see also \cite{operationalRMT} for a symbolic approach) allows us to compute the integrals $I_{-}$ with $p=1$
as 
\[
I_{-}(s+2,1)=\int_{0}^{\infty}x^{s-1}\frac{x^{2}}{\cosh x-\cos x}dx=\Gamma\left(s\right)g_{1}\left(-s\right),
\]

from which we deduce the result
\begin{proposition}
With $s$ a real number at least equal to $3$, we have
\[
I_{-}\left(s,1\right)=\int_{0}^{\infty}\frac{x^{s-1}}{\cosh x-\cos x}dx=2\Gamma\left(s\right)\left(\frac{i}{2}\right)^{\frac{s}{2}}\zeta_{2}\left(s,1;1-i,1+i\right).
\]
More generally,
\[
I_{-}\left(s,1,a,b\right):=\int_{0}^{\infty}\frac{x^{s-1}}{\cosh ax-\cos bx}dx=2\Gamma\left(s\right)\left(\frac{i}{2}\right)^{\frac{s}{2}}\zeta_{2}\left(s,a;a-i b,a+i b\right).
\]
\end{proposition}
\begin{proof}
From Lemma~\ref{lemma:expansion-lem}, we have
\[
\frac{x^{2}}{\cosh x-\cos x}=\sum_{p\ge0}\frac{x^{p}}{p!}g_{1}\left(p\right)
\]
with
\begin{align*}
g_{1}\left(p\right) & =\left(\sqrt{2i}\right)^{p}\left(\left(B_{1}+\frac{1}{2}\right)+i\left(B_{2}+\frac{1}{2}\right)\right)^{p}\\
&=\left(\sqrt{2i}\right)^{p}i^{p}\left(i\left(B_{1}+\frac{1}{2}\right)+\left(B_{2}+\frac{1}{2}\right)\right)^{p}\\
&=\left(\sqrt{\frac{2}{i}}\right)^{p}i^{2p}\left(i\left(B_{1}+\frac{1}{2}\right)+\left(B_{2}+\frac{1}{2}\right)\right)^{p}\\
&=\left(-1\right)^{p}\left(\sqrt{\frac{2}{i}}\right)^{p}\left(i\left(B_{1}+\frac{1}{2}\right)+\left(B_{2}+\frac{1}{2}\right)\right)^{p}\\
&=\left(-1\right)^{p}\left(\sqrt{\frac{2}{i}}\right)^{p}B_{p}^{\left(2\right)}\left(\frac{1+i}{2};1,i\right)
\end{align*}
so that 
\[
\frac{x^{2}}{\cosh x-\cos x}=\sum_{p\ge0}\left(-1\right)^{p}\frac{x^{p}}{p!}g\left(p\right)
\]
with
\[
g\left(p\right)=\left(\sqrt{\frac{2}{i}}\right)^{p}B_{p}^{\left(2\right)}\left(\frac{1+i}{2};1,i\right).
\]
The analytic continuation of the function $g$ can be found using
\cite[Eq. (7)]{bayad2013}
\[
B_{k+2}^{\left(2\right)}\left(x;a_{1},a_{2}\right)=a_{1}a_{2}\left(k+2\right)\left(k+1\right)\zeta_{2}\left(-k,x;a_{1},a_{2}\right)
\]
with the Barnes zeta function
\[
\zeta_{2}\left(s,x;a_{1},a_{2}\right)=\sum_{m_{1},m_{2}\ge0}\frac{1}{\left(x+a_{1}m_{1}+a_{2}m_{2}\right)^{s}}
\]
so that
\[
B_{-s}^{\left(2\right)}\left(x;a_{1},a_{2}\right)=a_{1}a_{2}s\left(s+1\right)\zeta_{2}\left(s+2,x;a_{1},a_{2}\right)
\]
and
\[
B_{-s}^{\left(2\right)}\left(\frac{1+i}{2};1,i\right)=i s\left(s+1\right)\zeta_{2}\left(s+2,\frac{1+i}{2};1,i\right).
\]
We therefore deduce
\begin{align*}
I_{-}\left(s+2,1\right)&=\int_{0}^{\infty}x^{s-1}\frac{x^{2}}{\cosh x-\cos x}dx\\
&=\Gamma\left(s\right)\left(\sqrt{\frac{i}{2}}\right)^{s}i s\left(s+1\right)\zeta_{2}\left(s+2,\frac{1+i}{2};1,i\right)\\
&=2\Gamma\left(s+2\right)\left(\sqrt{\frac{i}{2}}\right)^{s+2}\zeta_{2}\left(s+2,\frac{1+i}{2};1,i\right),
\end{align*}
from which the result follows.
\end{proof}

We study now, for an integer $p\ge1,$
\[
I_{-}\left(s,p\right)=\int_{0}^{\infty}\frac{x^{s-1}}{\left(\cosh x-\cos x\right)^{p}}dx.
\]

\begin{proposition}
The integral $I_{-}\left(s,p\right)$ is equal to
\[
I_{-}\left(s,p\right)=2^{p}\left(\sqrt{\frac{i}{2}}\right)^{s}\Gamma\left(s\right)\zeta_{2p}\left(s,p\frac{1+i}{2};\left(1,i\right)^{p}\right)=2^{p}\left(\sqrt{\frac{i}{2}}\right)^{s}\Gamma\left(s\right)\sum_{m,n\ge0}\frac{\binom{p-1+m}{m}\binom{p-1+n}{n}}{\left(p\frac{1+i}{2}+m+i n\right)^{s}}
\]
\end{proposition}

\begin{proof}
Since
\[
\frac{x^{2}}{\cosh x-\cos x}=\sum_{n\ge0}\frac{x^{n}}{n!}g_{1}\left(n\right)
\]
we have
\[
f^{p}\left(x\right)=\left(\frac{x^{2}}{\cosh x-\cos x}\right)^{p}=\sum_{n\ge0}\frac{x^{n}}{n!}g_{1}^{*p}\left(n\right)
\]
with the convolution
\[
g_{1}^{*p}\left(n\right)=\left(\sqrt{2i}\right)^{n}\left(\left(B_{1}^{\left(1\right)}+\frac{1}{2}\right)+i\left(B_{2}^{\left(1\right)}+\frac{1}{2}\right)+\dots+\left(B_{1}^{\left(p\right)}+\frac{1}{2}\right)+i\left(B_{2}^{\left(p\right)}+\frac{1}{2}\right)\right)^{n}
\]
so that
\begin{align*}
g_{1}^{*p}\left(n\right) & =\left(\sqrt{\frac{2}{i}}\right)^{n}\left(-1\right)^{n}\left(\left(B_{1}^{\left(1\right)}+\frac{1}{2}\right)+i\left(B_{2}^{\left(1\right)}+\frac{1}{2}\right)+\dots+\left(B_{1}^{\left(p\right)}+\frac{1}{2}\right)+i\left(B_{2}^{\left(p\right)}+\frac{1}{2}\right)\right)^{n}\\
 & =\left(\sqrt{\frac{2}{i}}\right)^{n}\left(-1\right)^{n}B_{n}^{\left(2p\right)}\left(p\frac{1+i}{2};\left(1,i\right)^{p}\right)
\end{align*}
with the notation $\left(1,i\right)^{p}=\left(1,i,1,i,\dots,1,i\right)$
so that 
\[
\left(\frac{x^{2}}{\cosh x-\cos x}\right)^{p}=\sum_{n\ge0}\left(-1\right)^{n}\frac{x^{n}}{n!}g^{\left(p\right)}\left(n\right)
\]
with
\[
g^{\left(p\right) }\left(n\right)=\left(\sqrt{\frac{2}{i}}\right)^{n}B_{n}^{\left(2p\right)}\left(p\frac{1+i}{2};\left(1,i\right)^{p}\right).
\]
The Bernoulli-Barnes polynomial can be expressed as the value of the
Barnes zeta function at negative index as
\[
\zeta_{2p}\left(-k,x;\mathbf{a}\right)=\left(-1\right)^{2p}\frac{k!}{\left(2p+k\right)!}\frac{B_{2p+k}^{\left(2p\right)}\left(x;\mathbf{a}\right)}{i^{p}}
\]
so that
\begin{align*}
B_{-s}^{\left(2p\right)}\left(x;\mathbf{a}\right) & =\left(-1\right)^{p}\left(-s\right)\left(-s-1\right)\dots\left(-s-2p+1\right)i^{p}\zeta_{2p}\left(s+p,x;\mathbf{a}\right)\\
 & =i^{p}\frac{\Gamma\left(s+2p\right)}{\Gamma\left(s\right)}\zeta_{2p}\left(s+p,x;\mathbf{a}\right).
\end{align*}
We deduce
\[
g^{\left(p\right)}\left(-s\right)=\left(\sqrt{\frac{i}{2}}\right)^{s}B_{-s}^{\left(2p\right)}\left(p\frac{1+i}{2};\left(1,i\right)^{p}\right)=i^{p}\left(\sqrt{\frac{i}{2}}\right)^{s}\frac{\Gamma\left(s+2p\right)}{\Gamma\left(s\right)}\zeta_{2p}\left(s+2p,p\frac{1+i}{2};\left(1,i\right)^{p}\right)
\]
and
\begin{align*}
I_{-}\left(s+2p,p\right) & =\int_{0}^{\infty}x^{s-1}\frac{x^{2p}}{\left(\cosh x-\cos x\right)^{p}}dx=\Gamma\left(s\right)g\left(-s\right)\\
 & =i^{p}\left(\sqrt{\frac{i}{2}}\right)^{s}\Gamma\left(s+2p\right)\zeta_{2p}\left(s+2p,p\frac{1+i}{2};\left(1,i\right)^{p}\right)
\end{align*}
and finally
\begin{align*}
I_{-}\left(s,p\right) 
 & =i^{p}\left(\sqrt{\frac{i}{2}}\right)^{s-2p}\Gamma\left(s\right)\zeta_{2p}\left(s,p\frac{1+i}{2};\left(1,i\right)^{p}\right)\\
 & =2^{p}\left(\sqrt{\frac{i}{2}}\right)^{s}\Gamma\left(s\right)\zeta_{2p}\left(s,p\frac{1+i}{2};\left(1,i\right)^{p}\right)
\end{align*}
Lastly, notice that the $2p-$variate zeta function
\[
\zeta_{2p}\left(s,p\frac{1+i}{2};\left(1,i\right)^{p}\right)=\sum_{m_{1},n_{1},\dots,m_{p},n_{p}\ge0}\left(p\frac{1+i}{2}+m_{1}+i n_{1}+\dots+m_{p}+i n_{p}\right)^{-s}
\]
is in fact a two-variables Dirichlet series since
\[
\zeta_{2p}\left(s,p\frac{1+i}{2};\left(1,i\right)^{p}\right)=\sum_{m,n\ge0}\frac{\binom{p-1+m}{m}\binom{p-1+n}{n}}{\left(p\frac{1+i}{2}+m+i n\right)^{s}}
\]
as a consequence of the counting function
\[
\#\{\left(n_{1},\dots,n_{p}\right)\in\mathbb{N}^p:n_{1}+\dots+n_{p}=n\}=\binom{p-1+n}{n}.
\]
\end{proof}
\end{document}